\documentclass[]{sn-jnl}

\usepackage{graphicx}%
\usepackage{multirow}%
\usepackage{amsmath,amssymb,amsfonts}%
\usepackage{amsthm}%
\usepackage{mathrsfs}%
\usepackage[title]{appendix}%
\usepackage{xcolor}%
\usepackage{textcomp}%
\usepackage{manyfoot}%
\usepackage{booktabs}%
\usepackage{algorithm}%
\usepackage{algorithmicx}%
\usepackage{algpseudocode}%
\usepackage{listings,enumerate,cite}%
\theoremstyle{thmstyleone}%
\newtheorem{theorem}{Theorem}%  meant for continuous numbers
\newtheorem{proposition}{Proposition}
\newtheorem{lemma}{Lemma}
\theoremstyle{thmstyletwo}%
\newtheorem{corollary}{Corollary}

\theoremstyle{definition} % Non-italic style for theorems, remarks, etc.
\newtheorem{remark}[theorem]{Remark}
\newtheorem{example}[theorem]{Example}

\theoremstyle{thmstylethree}%
\newtheorem{definition}{Definition}%
\newcommand{\vertiii}[1]{\left\vert\kern-0.25ex\left\vert\kern-0.25ex\left\vert #1\right\vert\kern-0.25ex\right\vert\kern-0.25ex\right\vert}

\newcommand{\al}{\alpha}
\newcommand{\be}{\beta}
\newcommand{\ga}{\gamma}

\newcommand {\R} {\mathbb R}
\newcommand {\N} {\mathbb N}
\newcommand {\Z} {\mathbb Z}

\newcommand{\iv}{^{-1}}

\usepackage{mathptmx}   %Times New Roman
\hypersetup{
	colorlinks=true,
	linkcolor=blue, % Couleur des liens internes
	citecolor=red, % Couleur des num?ros de la biblio dans le corps
	urlcolor=blue  } % Couleur des url
\begin{document}

%\author{
%\name{Doan Huu Hieu\textsuperscript{1,2},
% Nguyen Duy Cuong\textsuperscript{3}
%}
%\thanks{CONTACT Nguyen Duy Cuong. Email: ndcuong@ctu.edu.vn}
%	\thanks{Dedicated to the memory of Prof Alexander Rubinov, a teacher and friend}
%\affil{\textsuperscript{1} 
%Faculty of Mathematics and Computer Science, University of Science, Ho Chi Minh City, Vietnam
%}
%\affil{\textsuperscript{2} 
%Vietnam National University, Ho Chi Minh City, Vietnam
%}
%\affil{\textsuperscript{3} 
%Department of Mathematics, College of Natural Sciences, Can Tho University, Can Tho City, Vietnam
%}}
%\maketitle

\title[A von Neumann-Jordan Constant of Non-Normable Metrics]{A von Neumann-Jordan Constant of Non-Normable Metrics}

\author[1,2]{\fnm{Doan} \sur{Huu Hieu}}\email{huuhieudoan2k@gmail.com}

\author*[3]{\fnm{Nguyen} \sur{Duy Cuong}}\email{ndcuong@ctu.edu.vn}

\affil[1]{\orgdiv{Faculty of Mathematics and Computer Science}, \orgname{University of Science}, \city{Ho Chi Minh City},  \country{Vietnam}}

\affil[2]{\orgdiv{Vietnam National University}, \city{Ho Chi Minh City},  \country{Vietnam}}

\affil[3]{\orgdiv{Faculty of Mathematics}, \orgname{College of Natural Sciences, Can Tho University}, \city{Can Tho City},  \country{Vietnam}}

\abstract{The paper studies a generalized von Neumann-Jordan constant of non-normable metrics on vector spaces. 
To the best of our knowledge, all existing results of the von Neumann-Jordan constant and its generalizations have been established only in the normed setting. 
We identify reasonable conditions on non-normable metrics under which results known for norms remain valid.
Several examples and counterexamples are provided to justify the established results. 
The computation for a class of non-normable metrics on product spaces is also investigated. 
In particular, we give precise formulas for the generalized von Neumann-Jordan constant of $p$-metrics under a metric-type Clarkson inequality.
Comparisons with existing results are discussed throughout the paper whenever applicable.}



\maketitle

\section{Introduction}\label{S1}

In the pioneering work~\cite{JorNeu35}, Jordan and von Neumann proved that a norm on a vector space satisfies the parallelogram identity if and only if it is induced by an inner product.
To quantify how far a norm deviates from this characterization, the authors associated to each norm a constant, now known as the \textit{von Neumann-Jordan constant}, taking values in the interval $[1,2]$.
More precisely, this constant equals one if and only if the norm comes from an inner product, and the larger this constant is, the less the norm can be approximated by such a structure.
Not long thereafter, Clarkson~\cite{Cla36,Cla37} obtained explicit formulas for this constant for the standard norms in Lebesgue spaces and used it to study the uniform convexity of these spaces. 
The aforementioned results have now become classic and laid the foundation for subsequent developments in  Banach space geometry; see \cite{Bea82,JohLin01} and the references therein.

There are several important directions related to the von Neumann-Jordan constant. 
Some works have computed this constant for particular Banach spaces: Lebesgue-Bochner  \cite{KatTak98}, Lorentz sequence \cite{KatMal01}, and Cesàro sequence \cite{Sae10}.
The study of new geometric constants related to the von Neumann-Jordan constant has also attracted considerable attention \cite{CuiHuaHudKac15,WanCuiZha15,DhoPirSae03,YanLi10,Wan10,WanPan09}.
Some studies are devoted to  the applicability of this constant in studying geometric properties of Banach spaces such as uniform convexity \cite{KatTak97+}, uniform nonsquareness \cite{TakKat98}, normal structure \cite{JimLloSae06}, and fixed point results \cite{DhoDomKaeKaePan06,Sae06}.

We recall that the von Neumann-Jordan constant concerns geometric properties of a given norm rather than the topology induced by the norm. 
While topology is essential for the study of convergence, continuity, approximations and other properties, the constant indicates how closely a given norm behaves like an inner product.
To justify this claim, we note that on a finite dimensional space all norms are topologically equivalent; nevertheless, their geometric properties can differ substantially. 
For example, the von Neumann-Jordan constant of the sum and max norms equals two, whereas for the Euclidean norm it is one. 
The reader is referred to \cite{IkeKat14,Cuo25,SaiKatTak00,SaiKatTak00+} and the references therein for the computation of this constant for various norms in finite and infinite dimensional spaces.

Given a metric on a vector space, the  pair is called a \emph{metric linear space} if the metric is compatible with the vector space structure in the sense that the addition and scalar multiplication operations are continuous with respect to the metric topology and the corresponding product topologies \cite{Rol84}.
A metric on a vector space is called \emph{translation-invariant} (or \emph{shift-invariant} \cite{Don21}) if the distance between any two points is preserved under the translation by the same vector.
In \cite{Kak36}, Kakutani proved that every metric linear space admits a translation-invariant metric that is topologically equivalent to the original one.
Thus, from a topological viewpoint, it suffices to assume that the original metric is translation-invariant when studying metric linear spaces \cite{Alb79,Mab03}. 
Note that the translation-invariance implies the continuity of the addition operation.

It is standard that a norm on a vector space induces a metric by taking the norm of the difference between any two points.
On the other hand, given a metric on a vector space, one can define an associated one-variable function by fixing one argument at the zero vector.
 In the context of metric linear spaces equipped with a translation-invariant metric, such a function has been studied in the literature under different names such as \textit{quasinorm} \cite{Yos95} or $F$-\textit{norm} \cite{Rol84,Rol72}.
In these settings, this function is generally not a norm in the conventional sense as it may fail to satisfy the absolute homogeneity. 
When it does satisfy the requirements of a norm, the corresponding metric is called a \emph{normable} metric \cite{SinNar20}.

In a recent study \cite{CuiHuaHudKac15}, the authors introduced a generalized von Neumann-Jordan constant with an order ranging from one to infinity and used it to study the uniform nonsquare property and fixed point results in Banach spaces.
When the order equals two, this constant reduces to the conventional one. 
The generalized constant has been further investigated in \cite{WanCuiZha15,RahGun21,AmiKho24}.
To the best of our knowledge, all available studies concerning the von Neumann-Jordan constant and its generalizations have been established only in the normed setting. 
Following the idea initiated in \cite{CuiHuaHudKac15}, we attempt to systematically study a generalized von Neumann-Jordan constant for non-normable metrics on vector spaces.
Comparisons with existing results of this type are  discussed throughout the paper.

The compatibility between the topological and algebraic structures provides a convenient framework for various studies. 
However, in some situations not all continuity assumptions are required. 
For example, the Lyusternik-Graves theorem \cite[Theorem~5.2]{Don21} in the field of variational analysis requires only a vector space equipped with a translation-invariant metric and does not need the scalar multiplication continuity.
In view of this observation and the fact that the von Neumann-Jordan constant 
mainly concerns geometric behavior rather than the topology involved, our setting in this paper is very general: we consider a (not necessarily translation-invariant) metric on a vector space without imposing any continuity compatibility conditions.

If a metric is normable, then all existing results concerning the von Neumann-Jordan constant and its generalizations are applicable. 
In this paper, we focus on the case of non-normable metrics. 
Our results are obtained under reasonable assumptions, and
we demonstrate that previously known results of this type arise as particular cases of our framework.
We would like to note that this work is not a purely formal extension from the normed to the metric setting, but rather an attempt to identify the weakest possible conditions under which results known for norms remain valid in the metric framework, and consequently,
to have a deeper understanding of the existing results and techniques in the normed setting.
Throughout the paper, several examples and counterexamples are provided to justify the established results.
In each example, we also explicitly examine whether any compatibility condition holds between the topological and algebraic structures.

The next Section~\ref{S2} presents basic definitions and results used throughout the paper. 
In particular, we prove a list of equivalent conditions under which a translation-invariant metric on a vector space is normable.
Section~\ref{S3} studies a generalized von Neumann-Jordan constant for non-normable metrics. 
We establish lower and upper bounds for this constant under reasonable assumptions and demonstrate that several non-normable metrics satisfy our results.
In addition, we establish a metric-type parallelogram law extending the classical results of Clarkson and von Neumann from the normed to the metric setting.
Moreover, a list of equivalent parameterized formulations of the generalized constant is provided.
The final Section~\ref{S4} is dedicated to the product setting. 
We provide estimates (exact in some cases) for the constant associated with metrics on a product of vector spaces constructed via a collection of continuous convex functions  on the standard simplex of $\mathbb{R}^n$.
Explicit formulas of the constant for the $p$-metrics under a metric-type Clarkson inequality is also presented.

\section{Preliminaries}\label{S2}
In this paper, we assume that $X$ is a vector space over a field $K\subset \mathbb{R}$ and write $(X,K)$ when it is necessary to emphasize the underlying field.
When $K=\mathbb{R}$, we simply say that $X$ is a vector space. 
The zero vector in $X$ is denoted by $0_X$.
We denote by $\mathbb{N},\mathbb{Q},\R_+,\R$ and $\mathbb{C}$ the sets of positive integers, rational, nonnegative real, real, and complex numbers, respectively.
The vector spaces $\mathbb{R}$ and $\mathbb{C}$ are assumed to be equipped with the standard norms, and we write $\infty$ instead of $+\infty$.

A set $\mathcal{B} \subset \mathbb{R}^n$ is called a basis of $\mathbb{R}^n$ over $K$ if $\mathcal{B}$ is linearly independent over $K$ (i.e., every finite subset of $\mathcal{B}$ is linearly independent over $K$) and every $x \in \mathbb{R}^n$ can be expressed as a finite linear combination of elements of $\mathcal{B}$ \cite[p.~78]{Mar09}. 
For any $x \in \mathbb{R}^n$, such a representation is unique \cite[Lemma 4.1.1]{Mar09}. 
In particular, every basis $\mathcal{B}$ of $\mathbb{R}^n$ over $\mathbb{Q}$ has cardinality equal to the continuum \cite[Theorem 4.2.3]{Mar09}.

The next definition recalls some basic properties needed for the study \cite{Pen13}.
\begin{definition}
Let $X$ be a vector space, and $\varphi: X\to\R$.
The function $\varphi$ is
\begin{enumerate}[\rm (i)]
\item\label{D1.1-1}
\textit{midpoint convex} if $\varphi\left(\frac{x+y}{2}\right)\le\frac{\varphi(x)+\varphi(y)}{2}$
for all $x,y\in X$;	
\item \label{D1.1-2}
\textit{convex} if $\varphi(\lambda x+(1-\lambda)y)\le \lambda \varphi(x)+(1-\lambda)\varphi(y)$  for all $x,y\in X$ and $\lambda\in[0,1]$;
\item \label{D1.1-3}
\textit{even} if $\varphi(-x)=\varphi(x)$ for all $x\in X$;
\item \label{D1.1-4}
$\lambda$-\textit{homogeneous} for some $\lambda\in\mathbb{R}$ if $\varphi(\lambda x)=\lambda \varphi(x)$ for all $x\in X$;
\item \label{D1.1-5}
\textit{positively homogeneous} if $\varphi(\lambda x)=\lambda \varphi(x)$ for all $x\in X$ and $\lambda\ge 0$;
\item \label{D1.1-6}
\textit{absolutely homogeneous} if $\varphi(\lambda x)=|\lambda|\varphi(x)$ for all $x\in X$ and $\lambda\in\R$;
\item \label{D1.1-7}
\textit{subadditive} if $\varphi(x+y)\le \varphi(x)+\varphi(y)$ for all $x,y\in X$;
\item \label{D1.1-8}
\textit{additive} if $\varphi(x+y)= \varphi(x)+\varphi(y)$ for all $x,y\in X$.
\end{enumerate}
\end{definition}	

\begin{remark}\label{R1.2}
\begin{enumerate}[\rm (i)]
\item\label{R1.2-1}
It is clear that \eqref{D1.1-2} $\Rightarrow$ \eqref{D1.1-1}. 
The converse implication holds when $\varphi$ is continuous
\cite{Jen06,TamTsu21}. 
If $\varphi$ is midpoint convex and 2-homogeneous, then it is subadditive.
The function $\varphi$ is absolutely homogeneous if and only if it is even and positively homogeneous.
\item\label{R1.2-2}
The function $\varphi$ is $\lambda$-homogeneous for some $\lambda\ne 0$ if and only if  $\varphi(\lambda^kx)=\lambda^k\varphi(x)$ for all $x\in X$ and integers $k$.
The statement is obvious for nonnegative integers.
If $k$ is a negative integer, then $\varphi(\lambda^kx)=\lambda^k\cdot\lambda^{-k}\varphi(\lambda^kx)=\lambda^k\varphi(\lambda^{-k}\lambda^kx)=\lambda^k\varphi(x)$ for all $x\in X$.
\item\label{R1.2-3}
If $\varphi:\R^n\to \R$ is additive, then $\varphi$ is  $\lambda$-homogeneous for any $\lambda\in\mathbb{Q}$; see \cite[Theorem 5.2.1]{Mar09}.
\end{enumerate}
\end{remark}

Let $d$ be a metric on a vector space $X$.
From now on, we consider the function $f:X\to\mathbb{R}_+$ given by
\begin{gather}\label{f}
f(x):=d(x,0_X)\quad \text{for all } x\in X.
\end{gather}
If $f$ is a norm, then $d$ is said to be {normable} \cite{SinNar20}; otherwise, $d$ is called {non-normable}.
The metric $d$ is called {translation-invariant} if $d(x+z,y+z)=d(x,y)$ for all $x,y,z\in X$.

\begin{proposition}\label{P2.3}
The following assertions hold:
\begin{enumerate}[\rm (i)]
\item\label{P2.3-1}
the function $f$ is 
%\todo{nonexpansive}
Lipschitz continuous with modulus 1;
\item\label{P2.3-2}
if $d$ is translation-invariant, then $f$ is subadditive and even.
\end{enumerate}	
\end{proposition}	

\begin{proof}
\begin{enumerate}[\rm (i)]
\item
By the triangle inequality, $d(x,0_X)\le d(x,y)+d(y,0_X)$ and $d(y,0_X)\le d(y,x)+d(x,0_X)$, and consequently,
$|f(x)-f(y)|\le d(x,y)$ for all $x,y\in X$.
Observe that
$|f(x)-f(0_X)| = d(x,0_X)$ for any $x\in X$.
This justifies the statement.
\item 
By the translation-invariance, $f(-x)=d(-x,0_X)=d(0_X,x)=d(x,0_X)=f(x)$ for all $x\in X$ and $f(x+y)=d(x+y,0_X)\le d(x+y,y)+d(y,0_X)=f(x)+f(y)$ for all  $x,y\in X$.
\end{enumerate}		
The proof is complete.
\end{proof}	

\begin{theorem}\label{P5.3}
Suppose that $d$ is translation-invariant.
The following assertions are equivalent:
\begin{enumerate}[\rm (i)]
\item\label{P5.3-1}
$f$ is midpoint convex;
\item\label{P5.3-2}
$f$ is convex;
\item\label{P5.3-3}
$f(\lambda x)\le \lambda f(x)$ for all $x\in X$ and 
$\lambda\in[0,1]$;
\item\label{P5.3-4}
$f\left(\frac{x}{2}\right)\le\frac{f(x)}{2}$ for all $x\in X$;
\item\label{P5.3-6}
$f$ is positively homogeneous;
\item\label{P5.3-5}
$f$ is absolutely homogeneous.
\end{enumerate}	
\end{theorem}	

\begin{proof}
By Proposition~\ref{P2.3}\eqref{P2.3-1}, $f$ is continuous.	
It is clear that \eqref{P5.3-1} $\Leftrightarrow$ \eqref{P5.3-2} $\Rightarrow$ \eqref{P5.3-3} $\Rightarrow$ \eqref{P5.3-4}.
By Proposition~\ref{P2.3}\eqref{P2.3-2},  $f$ is subadditive and even.
If condition \eqref{P5.3-4} is satisfied, thanks to the subadditivity, $f$ is midpoint convex.
Thus, \eqref{P5.3-1}-\eqref{P5.3-4} are equivalent.
The evenness implies that \eqref{P5.3-6} and \eqref{P5.3-5} are equivalent.

We now show that \eqref{P5.3-3} and \eqref{P5.3-6} are equivalent.
It is clear that \eqref{P5.3-6} $\Rightarrow$ \eqref{P5.3-3}.
Assume that  assertion \eqref{P5.3-3} is satisfied.
Let $x\in X$ and $\lambda\ge 0$.
If $\lambda>1$, then $f(x)=f\left(\frac{1}{\lambda}\cdot\lambda x\right)\le \frac{f(\lambda x)}{\lambda}$, and consequently, $\lambda f(x)\le f(\lambda x)$.
We have $\lambda=m+\xi$ for some $m\in\N$ and $\xi\in(0,1)$.
By the subadditivity,
\begin{align*}
f(\lambda x)
=f(mx+\xi x)
\le mf(x)+f(\xi x)\le m f(x)+\xi f(x)=\lambda f(x).
\end{align*}	
Hence, $f(\lambda x)=\lambda f(x)$.
If $\lambda\in(0,1]$, we only need to show that $f(\lambda x)\ge\lambda f(x)$.
Suppose that $f(\lambda x)<\lambda f(x)$.
By the above argument, $f(x)=f\left(\frac{1}{\lambda}\cdot\lambda x\right)
=\frac{f(\lambda x)}{\lambda}<\frac{\lambda f(x)}{\lambda}=f(x),$
a contradiction.
Thus, condition \eqref{P5.3-6} is satisfied.
This completes the proof.
\end{proof}	

\begin{remark}
The equivalence between \eqref{P5.3-2}, \eqref{P5.3-3} and \eqref{P5.3-6} is studied in \cite[Lemmas~2 \& 3]{SinNar20}.
\end{remark}	

\begin{proposition}\label{P1.5}
Let $d$ be a metric on a vector space $X$, and $f$ be given by \eqref{f}.
The function $f$ is a norm on $X$ if and only if it is midpoint convex and absolutely homogeneous.
\end{proposition}	

\begin{proof}
The direct implication is obvious.
Suppose that $f$ is midpoint convex and 
absolutely homogeneous.
We only need to verify the triangle inequality since the other conditions are obviously satisfied.
We have
\begin{align*}
f(x+y)
=f\left(2\cdot\dfrac{x+y}{2}\right)
=2f\left(\dfrac{x+y}{2}\right)
\le 2\cdot\frac{f(x)+f(y)}{2}=f(x)+f(y).
\end{align*}	
for all $x,y\in X$.
Thus, $f$ is a norm on $X$.
\end{proof}	

\begin{remark}\label{R2.7}
\begin{enumerate}[\rm (i)]
\item\label{R2.7-1}
A result analogous to Proposition~\ref{P1.5} can be found in \cite[Theorem~4]{SinNar20}.
\item\label{R2.7-2}
In view of Proposition~\ref{P1.5}, if $d$ is translation-invariant and any of the conditions \eqref{P5.3-1}-\eqref{P5.3-5} in Theorem~\ref{P5.3} is satisfied, then $f$ is a norm on $X$.
In Sections~\ref{S3} and~\ref{S4}, we study the function $f$ under certain reasonable conditions without assuming the translation-invariance.
\end{enumerate}
\end{remark}	

The following elementary result is needed for the subsequent analysis.
Since we were unable to find a reference in the literature, we provide a proof for the sake of completeness.
\begin{lemma}\label{L1.5}
Let $\al>0$ and $n\in\N$.
The following assertions hold: 	
\begin{enumerate}[\rm (i)]
\item\label{L1.5-1}
if $\al\in(0,1]$, then 
\begin{gather*}
\left(\sum_{i=1}^{n}a_i\right)^\al\le \sum_{i=1}^{n}a^\al_i\le n^{1-\al}\left(\sum_{i=1}^{n}a_i\right)^\al
\end{gather*}	
for all $a_1,\ldots,a_n\ge 0$;
\item\label{L1.5-2}
 if $\al\in[1,\infty)$, then
 \begin{gather*}
 \sum_{i=1}^{n}a^\al_i\le \left(\sum_{i=1}^{n}a_i\right)^\al\le n^{\al-1}\sum_{i=1}^{n}a^\al_i
 \end{gather*}	
for all $a_1,\ldots,a_n\ge 0$;
\item\label{L1.5-3}
if $\al\in(0,1]$, then 
\begin{gather*}
\left(\sum_{i=1}^{n}(a_i+b_i)^\alpha\right)^\frac{1}{\alpha} \ge \left(\sum_{i=1}^{n}a_i^\alpha\right)^\frac{1}{\alpha}+\left(\sum_{i=1}^{n}b_i^\alpha\right)^\frac{1}{\alpha}
\end{gather*}	
 for all $a_1,\ldots,a_n,b_1,\ldots,b_n\ge 0$.
\end{enumerate}	
\end{lemma}	 

\begin{proof}
\begin{enumerate}[\rm (i)]
\item	
Let $\alpha\in(0,1]$. 
We prove the first inequality by induction.
It is clear that the inequality holds for $n=1$.
Suppose that it holds for $n=k$ for some $k\in\mathbb{N}$.
The  decreasing property of  $h(t):=(1+t)^\al-1-t^\al$ 
on $\R_+$ implies that
$(1+t)^\al\le 1+t^\al$ for all $t\ge 0$.
Let $a_1,\ldots,a_{k+1}\ge 0$.
If $a_{k+1}=0$, then the inequality  follows from the induction hypothesis.
If $a_{k+1}>0$, then set $t:=\frac{\sum_{i=1}^{k}a_i}{a_{k+1}}$, and consequently,
\begin{align*}
\left(\sum_{i=1}^{k+1}a_i\right)^\al
= a_{k+1}^\alpha (1+t)^\alpha
 \le a_{k+1}^\alpha(1+t^\alpha) = \left(\sum_{i=1}^{k}a_i\right)^\al+a_{k+1}^\alpha\le \left(\sum_{i=1}^{k+1}a_i\right)^\al.
\end{align*}	
This completes the proof of the first inequality.
For the second inequality, it suffices to note that
the function
$t\mapsto t^\al$ is concave on $\R_+$, and consequently,
\begin{gather*}
\frac{\sum_{i=1}^{n}a_i^\al}{n}\le \left(\frac{\sum_{i=1}^{n}a_i}{n}\right)^\al
\Leftrightarrow
 \sum_{i=1}^{n}a^\al_i\le n^{1-\al}\left(\sum_{i=1}^{n}a_i\right)^\al.
\end{gather*}	
\item 
Let  $\al\in[1,\infty)$.
Then the function $h$ is increasing on $\R_+$.
Proceeding as in the above proof, we obtain the first inequality.
The convexity of the function
$t\mapsto t^\al$ on $\R_+$ implies that 
\begin{gather*}
\left(\frac{\sum_{i=1}^{n}a_i}{n}\right)^\al\le 	
\frac{\sum_{i=1}^{n}a_i^\al}{n}
\Leftrightarrow
\left(\sum_{i=1}^{n}a_i\right)^\al\le	n^{\al-1}
\left(\sum_{i=1}^{n}a_i\right)^\al,
\end{gather*}	
which proves the second inequality.
\item
Let $\mathrm{A}:=\left(\sum_{i=1}^{n}a_i^\alpha\right)^\frac{1}{\alpha}$ and $\mathrm{B}:=\left(\sum_{i=1}^{n}b_i^\alpha\right)^\frac{1}{\alpha}$.
If either $\mathrm{A}=0$ or $\mathrm{B}=0$, then the inequality is obviously satisfied.
Suppose that  $\mathrm{A}\ne0$ and $\mathrm{B}\ne0$. 
Let $\gamma:=\frac{\mathrm{A}}{\mathrm{A}+\mathrm{B}}\in (0,1)$. 
The concavity of the function
$t\mapsto t^\al$  on $\R_+$ implies that
\begin{gather*}
\left(\gamma\cdot\frac{a_i}{\gamma}+(1-\gamma)\cdot\frac{b_i}{1-\ga}\right)^\alpha
\ge\gamma\cdot \frac{a_i^\alpha}{\gamma^\alpha}+(1-\gamma)\cdot\frac{b_i^\alpha}{(1-\gamma)^\alpha}\;\;(i=1,\ldots,n).
\end{gather*}	
Thus,
\begin{align*}
\sum_{i=1}^{n}(a_i+b_i)^\alpha
&=\sum_{i=1}^{n}\left(\gamma\cdot\frac{a_i}{\gamma}+(1-\gamma)\cdot\frac{b_i}{1-\gamma}\right)^\alpha\\
&\ge \sum_{i=1}^{n}\left(\gamma\cdot \frac{a_i^\alpha}{\gamma^\alpha}+(1-\gamma)\cdot\frac{b_i^\alpha}{(1-\gamma)^\alpha}\right)\\
&=\gamma\cdot\frac{\mathrm{A}^\alpha}{\gamma^\alpha}+(1-\gamma)\cdot\frac{\mathrm{B}^\alpha}{(1-\gamma)^\alpha}\\
&=\gamma\cdot (\mathrm{A}+\mathrm{B})^\alpha+(1-\gamma)\cdot (\mathrm{A}+\mathrm{B})^\alpha\\
&=(\mathrm{A}+\mathrm{B})^\alpha,
\end{align*}
and consequently,
$\mathrm{A}+\mathrm{B}\le\left(\sum_{i=1}^{n}(a_i+b_i)^\alpha\right)^{\frac{1}{\al}}$.
\end{enumerate}
The proof is complete.
\end{proof}

The following result is needed for the construction of examples and counterexamples in Section~\ref{S3}; see \cite[Theorems~4.2.1 and~5.2.2, Corollary~5.2.2]{Mar09}.
\begin{lemma}\label{L2.10}
Let $(\mathbb{R}^n,\mathbb{Q})$ be a vector space.
The following assertions hold:
\begin{enumerate}[\rm (i)]
\item\label{L2.10-1}
if $S \subset \mathbb{R}^n$ is a linearly independent set over $\mathbb{Q}$, then
there exists a basis $\mathcal{B}\supset S$ of $\mathbb{R}^n$ over $\mathbb{Q}$;	
\item\label{L2.10-2}
if $\mathcal{B}$ is a  basis of $\R^n$ over $\mathbb{Q}$ and $g:\mathcal{B}\to \R$, then there exists a unique additive function $\varphi: \R^n\to \R$ such that $\varphi(x)=g(x)$ for all $x\in\mathcal{B}$. 
Moreover, if $g(\mathcal{B})\subset\mathbb{Q}$, then $\varphi$  is discontinuous on $\R^n$.
\end{enumerate}
\end{lemma}

\section{A von Neumann-Jordan constant of non-normable metrics}\label{S3}
In this section, we assume that $d$ is a metric on a vector space $X$, the function $f$ is given by \eqref{f}, and $\sigma\in[1,\infty)$.
Define
 $\mathcal{G}_{d}^{(\sigma)}:X\times X\to(0,\infty)$  by
\begin{gather}\label{G}
\mathcal{G}^{(\sigma)}_{d}(x,y):=\dfrac{f^\sigma(x+y)+f^\sigma(x-y)}{2^{\sigma-1}(f^\sigma(x)+f^\sigma(y))}
\end{gather}	
for all $x,y\in X$ with $(x,y)\ne (0_X,0_X)$.
The von Neumann-Jordan constant of order $\sigma$ associated with the metric $d$ is defined by
\begin{gather}\label{C1}
C_{\text{NJ}}^{(\sigma)}(d):=\sup_{(x,y)\ne (0_X,0_X)}\mathcal{G}_{d}^{(\sigma)}(x,y).
\end{gather}	

\begin{remark}\label{R2.1}
\begin{enumerate}[\rm (i)]
\item
The constant \eqref{C1} was originally studied in \cite{CuiHuaHudKac15} when $f$ is a norm, and
it coincides with the conventional von Neumann-Jordan constant when $\sigma = 2$.
It is worth noting that in our setting $d$ is generally a non-translation-invariant and non-normable metric.
\item
When $f$ is a norm, the above constant has been employed to study fixed point results, uniform nonsquareness, and normal structure in Banach spaces \cite{CuiHuaHudKac15, WanCuiZha15}. 
Explicit computations of this constant for specific Banach spaces can be found in \cite{RahGun21, AmiKho24}.
\end{enumerate}
\end{remark}

The following statement gives lower and upper bounds for the  constant.
\begin{theorem}\label{T2.5}
Let $\sigma\in[1,\infty)$.
The following assertions hold:
\begin{enumerate}[\rm (i)]
\item\label{T2.5-1}
$C_{\text{\rm NJ}}^{(\sigma)}(d)\ge 2^{2-\sigma}$;
\item\label{T2.5-2}
if $f$ is 2-homogeneous, then $C_{\text{\rm NJ}}^{(\sigma)}(d)\ge 1$;
\item\label{T2.5-3}
if $f$ is subadditive and even, then
$C_{\rm{NJ}}^{(\sigma)}(d)\le 2$.
\end{enumerate}		
\end{theorem}

\begin{proof}
Assertion \eqref{T2.5-1} is a consequence of the fact that $\mathcal{G}^{(\sigma)}_{d}(x,0_X)=2^{2-\sigma}$ for all $x\in X$ with $x\ne 0_X$.
If $f$ is 2-homogeneous, then
\begin{gather*}
\mathcal{G}^{(\sigma)}_{d}(x,x)
=\dfrac{f^\sigma(2x)}{2^{\sigma}f^\sigma(x)}= \dfrac{2^{\sigma}f^\sigma(x)}{2^{\sigma}f^\sigma(x)}=1
\;\;\text{for any}\;\;x\ne 0_X.
\label{P4.5-5}
\end{gather*}	
Thus, assertion \eqref{T2.5-2} is satisfied.
Suppose that $f$ is subadditive and even.
Let $x,y\in X$.
Then $f(x\pm y)\le f(x)+f(y)$.
By the second inequality in Lemma~\ref{L1.5}\eqref{L1.5-2} with $\al:=\sigma$ and $n=2$,
\begin{gather*}
	f^\sigma(x+y)+f^\sigma(x-y)\le 
	2(f(x)+f(y))^\sigma\le 2^\sigma(f^\sigma(x)+f^\sigma(y)).
\end{gather*}	
By \eqref{G} and \eqref{C1}, $C_{\text{NJ}}^{(\sigma)}(d)\le 2$.
This completes the proof.
\end{proof}	

\begin{remark}
\begin{enumerate}[\rm (i)]
\item		
When $f$ is a norm (in which case all the above assumptions are satisfied), Theorem~\ref{T2.5} recaptures \cite[Theorem~3.1]{CuiHuaHudKac15}.
%We emphasize that the above assumptions are essential for ensuring these lower and upper bounds. Several examples and counterexamples illustrating this necessity are provided below.
\item 
In view of assertion~\eqref{T2.5-1}, the statement in assertion~\eqref{T2.5-2} is meaningful only when $\sigma \in [2,\infty)$.
\end{enumerate}
\end{remark}	

The inequality in Theorem~\ref{T2.5}\eqref{T2.5-1} is not strict.
\begin{example}\label{Exam2.3}
Let $X$ be a Hilbert space, and
$d(x,y):=\min\{\|x-y\|,1\}$ for all $x,y\in X$.
It is clear that $d$ is a metric, and
$f(x):=d(x,0_X)=\min\{\|x\|,1\}$ for all $x\in X$. 
Observe that $f$ is not a norm since 
%\todo{$f(2\mathbf{e})=1\ne 2=2f(\mathbf{e})$}
$f(2\mathbf{e})=1\ne 2=2f(\mathbf{e})$, where $\mathbf{e}\in X$ is a unit vector.
Let $\sigma\in [1,2]$.
We are going to show that $C_{\text{NJ}}^{(\sigma)}(d)
\le 2^{2-\sigma}$. 
Let $x,y\in X$ with $(x,y)\ne (0_X,0_X)$.
Since $X$ is a Hilbert space, we have
\begin{gather*}
\|x+y\|^2+\|x-y\|^2=2(\|x\|^2+\|y\|^2).
\end{gather*}	
From this and the second inequality in Lemma~\ref{L1.5}\eqref{L1.5-1} with $\al:=\frac{\sigma}{2}\in[1/2,1]$ and 
$n=2$, 
\begin{align*}
\|x+y\|^\sigma+\|x-y\|^\sigma&=(\|x+y\|^2)^{\frac{\sigma}{2}}+(\|x-y\|^2)^{\frac{\sigma}{2}}\\
&\le 2^{1-\frac{\sigma}{2}}(\|x+y\|^2+\|x-y\|^2)^{\frac{\sigma}{2}}
=2\left(\|x\|^2+\|y\|^2\right)^{\frac{\sigma}{2}}
\le 2(\|x\|^\sigma+\|y\|^\sigma).
\end{align*}
Note that $f$ takes values on $[0,1]$.
If $\|x\|<1$ and $\|y\|<1$, then
\begin{eqnarray*}
f^\sigma(x+y)+f^\sigma(x-y)\le \|x+y\|^\sigma+\|x-y\|^\sigma
\le2(\|x\|^\sigma+\|y\|^\sigma)
=2(f^\sigma(x)+f^\sigma(y)),
\end{eqnarray*}	
and consequently, $\mathcal{G}^{(\sigma)}_{d}(x,y)\le 2^{2-\sigma}.$
If either $\|x\|\ge 1$ or $\|y\|\ge 1$, then 
\begin{gather*}
\mathcal{G}^{(\sigma)}_{d}(x,y)\overset{\eqref{G}}{=}\dfrac{f^\sigma(x+y)+f^\sigma(x-y)}{2^{\sigma-1}(f^\sigma(x)+f^\sigma(y))} \le \dfrac{f^\sigma(x+y)+f^\sigma(x-y)}{2^{\sigma-1}}  \le \frac{2}{2^{\sigma-1}}=2^{2-\sigma}.
\end{gather*}	
By Theorem~\ref{T2.5}\eqref{T2.5-1},  $C_{\text{NJ}}^{(\sigma)}(d)=2^{2-\sigma}$.
It is straightforward to verify that  convergence with respect to  $d$ is equivalent to convergence in $\|\cdot\|$.
Therefore, the addition and scalar multiplication operations are continuous.
\end{example}

The next example shows that there exists a non-normable metric satisfying all conditions in parts \eqref{T2.5-2} and \eqref{T2.5-3} of Theorem~\ref{T2.5}.
\begin{example}\label{Exam2.5}
Let $(\R^n,\R)$ be a vector space equipped with a norm $\|\cdot\|$, $\mathbf{e}\in\R^n$ be a unit vector,  and $S:=\{\mathbf{e},\sqrt{2}\mathbf{e}\}$.
Consider the restriction $(\R^n,\mathbb{Q})$ of $(\R^n,\R)$.
Observe that $S$ is a linearly independent set in $(\R^n,\mathbb{Q})$. 
By Lemma~\ref{L2.10}\eqref{L2.10-1}, there exists a  basis $\mathcal{B}\supset S$ of $(\R^n,\mathbb{Q})$.
Define $g:\mathcal{B}\to\R$  by
\begin{gather*}
g(x):=\begin{cases}
0  & \text{\rm if } x\in \mathcal{B}\setminus\{\mathbf{e}\},\\
1 & \text{\rm if } x= \mathbf{e}.
\end{cases} 
\end{gather*}
By Lemma~\ref{L2.10}\eqref{L2.10-2}, there exists a unique additive function $\varphi:\R^n\to \R$  such that $\varphi(x)=g(x)$ for all $x\in\mathcal{B}$.
By Remark~\ref{R1.2}\eqref{R1.2-3}, $\varphi$ is $\lambda$-homogeneous for any $\lambda\in\mathbb{Q}$.
In particular, $\varphi(0)=0$ and $\varphi(-x)=-\varphi(x)$ for all $x\in \R^n$.
Define 
\begin{gather}\label{E3.5-1}
d(x,y):=\|x-y\|+\big|\varphi(x-y)\big|\;\;\text{for all}\;\;x,y\in\R^n.
\end{gather}	
Observe that $d$ is symmetric, nonnegative and $d(x,y)=0$ if and only if $x=y$.
The triangle inequality of $d$ is a consequence of that of the norm and the additivity of $\varphi$.
Thus, $d$ is a metric on $\R^n$.
The function
%\todo{$f(x):=d(x,0_{\R^n})$}
 $f(x):=d(x,0_{\R^n})=\|x\|+|\varphi(x)|$ for all $x\in\R^n$
is subadditive, even, and $\lambda$-homogeneous for all nonnegative $\lambda\in\mathbb{Q}$. 
However, it is not a norm on $(\R^n,\R)$.
Indeed, $f(\sqrt 2\mathbf{e})
=\|{\sqrt 2}{\mathbf{e}}\|+|\varphi(\sqrt{2}\mathbf{e})|=\sqrt{2}+0=\sqrt{2}$
and $f(\mathbf{e})=\|\mathbf{e}\|+|\varphi(\mathbf{e})|=1+1=2$, and consequently, $f(\sqrt{2}\mathbf{e}) \ne \sqrt 2 f(\mathbf{e})$.\\
\textbf{Claim 1.} $\sup_{0<\|x\|\le 1}|\varphi(x)|=\infty$. 
Observe that $g(\mathcal{B})\subset\mathbb{Q}$.
By Lemma~\ref{L2.10}\eqref{L2.10-2}, $\varphi$ is discontinuous on $\R^n$. 
Suppose that $|\varphi(x)|\le M$ for all $x\in\R^n$ with $0<\|x\|\le 1$ and some $M>0$. 
Let $x\ne 0$.
Since $\mathbb{Q}$ is dense in $\R$, there exists a $q\in\mathbb{Q}$ such that $\frac{1}{2\|x\|}<q<\frac{1}{\|x\|}$.
Then $|\varphi(x)|=\frac{1}{q}|\varphi(qx)|\le \frac{M}{q}<2M\|x\|\to 0$ as $x\to 0$, which contradicts the discontinuity of $\varphi$.\\
\textbf{Claim 2.}
$\varphi(x)\in\mathbb{Q}$ for all $x\in\mathbb{R}^n$.
Indeed, for any $x\in\R^n$, there exist an integer $m\in\N$, scalars
$q_1,\ldots,q_m\in \mathbb{Q}$ and vectors $\mathbf{b}_1,\ldots,\mathbf{b}_m\in\mathcal{B}$ such that $x=\sum_{i=1}^{m}q_i\mathbf{b}_i$. 
By the additivity,
\begin{gather*}		\varphi(x)=\varphi\left(\sum_{i=1}^{m}q_i\mathbf{b}_i\right)=\sum_{i=1}^{m}q_i\varphi(\mathbf{b}_i)=\sum_{i=1}^{m}q_ig(\mathbf{b}_i)\in\mathbb{Q}.
\end{gather*}
\textbf{Claim 3.}  $C_{\text{NJ}}^{(\sigma)}(d)=2$ for any   $\sigma\in [1,\infty)$.
By Claims 1 and 2, there exists a sequence $\{x_k\}\subset\R^n$ such that $0<\|x_k\|\leq 1$ for all $k\in\N$ and the rational sequence $M_k:=|\varphi(x_k)|\to\infty$ as $k\to\infty$.
Let $y_k:=-x_k+\varphi(x_k)\cdot\mathbf{e}$ for all $k\in\N$.
Then
\begin{gather*}
\varphi(y_k)=0,\;\varphi(x_k+y_k)=\varphi(x_k),\;\varphi(x_k-y_k)=\varphi(2x_k-\varphi(x_k)\cdot\mathbf{e})=\varphi (x_k),
\end{gather*}	
and consequently,
\begin{gather*}
f(x_k)=\|x_k\|+|\varphi(x_k)|\leq 1+M_k,\;
f(y_k)
=\|-x_k+\varphi(x_k)\cdot\mathbf{e}\|\le 1+M_k,\\
f(x_k+y_k)=2M_k,\;
f(x_k-y_k)=\|2x_k-\varphi(x_k)\cdot\mathbf{e}\|+|\varphi(x_k)|\ge 2M_k-2
\end{gather*}
for all $k\in\mathbb{N}$.
By \eqref{G} and \eqref{C1},
\begin{align*}
C_{\text{NJ}}^{(\sigma)}(d)\ge 
\frac{f^\sigma(x_k+y_k)+f^\sigma(x_k-y_k)}{2^{\sigma-1}\left(f^\sigma(x_k)+f^\sigma(y_k)\right)}
\ge\frac{(2M_k)^\sigma+(2M_k-2)^\sigma}{2^{\sigma-1}\left((1+M_k)^\sigma+(1+M_k)^\sigma\right)}:=\gamma_k
\end{align*}
for all $k\in\N$.
We have
\begin{align*}
\lim_{k\to\infty}\gamma_k= \lim_{k\to\infty}\frac{1+\left(1-\frac{1}{M_k}\right)^\sigma}{\left(1+\frac{1}{M_k}\right)^\sigma}=2.
\end{align*}
Thus, $C_{\text{NJ}}^{(\sigma)}(d)\ge 2$.
By Theorem~\ref{T2.5}\eqref{T2.5-3}, we have $C_{\text{NJ}}^{(\sigma)}(d)=2$.
We now show that only the addition operation is continuous.
Let $x,y\in \R^n$ and $\{x_k\}, \{y_k\}\subset \R^n$ such that 
$d(x_k,x)\to 0$ and $d(y_k,y)\to 0$ as $k\to\infty$.
By the additivity of $\varphi$,  
\begin{align*}
d(x_k+y_k,x+y)&
\overset{\eqref{E3.5-1}}{=}\|x_k-x+y_k-y\|+\big|\varphi(x_k-x+y_k-y)\big|\\
&\leq \|x_k-x\|+\|y_k-y\|+\big|\varphi(x_k-x)+\varphi(y_k-y)\big|\\
&\leq \|x_k-x\|+\big|\varphi(x_k-x)\big|+\|y_k-y\|+\big|\varphi(y_k-y)\big|\\
&\overset{\eqref{E3.5-1}}{=}d(x_k,x)+d(y_k,y)\to 0 \;\;\text{as}\;\; k\to\infty.
\end{align*}
On the other hand, let $z_k:=\mathbf{e}$ $(k\in\N)$ and
$\{\lambda_k\}\subset\mathbb{Q}$ with $\lambda_k\to\sqrt{2}$ as $k\to\infty$. 
By the additivity and homogeneity of $\varphi$,  
\begin{align*}
d(\lambda_k z_k,\sqrt{2}\mathbf{e})
&\overset{\eqref{E3.5-1}}{=}\|(\lambda_k-\sqrt{2})\mathbf{e}\|+\big|\varphi(\lambda_k\mathbf{e}-\sqrt{2}\mathbf{e})\big|\\
&=\|(\lambda_k-\sqrt{2})\mathbf{e}\|+\big|\varphi(\lambda_k\mathbf{e})-\varphi(\sqrt{2}\mathbf{e})\big|\\
&=|\lambda_k-\sqrt{2}|+\big|\lambda_k\varphi(\mathbf{e})\big|
=|\lambda_k-\sqrt{2}|+\big|\lambda_k\big|\to\sqrt{2}\;\;\text{as}\;\; k\to\infty.
\end{align*}
Thus, the scalar multiplication operation is not continuous.
\end{example}

The $2$-homogeneity condition in Theorem~\ref{T2.5}\eqref{T2.5-2} is essential.
\begin{example}\label{E3.3}
Let $(X,\|\cdot\|)$ be a normed space, and
$d(x,y):=\|x-y\|^{\frac{1}{5}}$ for all $x,y\in X$.
Observe that $d$ is a metric on $X$.
Indeed, it suffices to verify the triangle inequality since the other conditions are obviously satisfied.
In view of the first inequality in  Lemma~\ref{L1.5}\eqref{L1.5-1} with $\al:=\frac{1}{5}$ and $n=2$, 
\begin{gather*}
\|x-y\|^{\frac{1}{5}}\le(\|x-z\|+\|y-z\|)^{\frac{1}{5}}
\le \|x-z\|^{\frac{1}{5}}+\|y-z\|^{\frac{1}{5}}
\;\;\text{for all}\;\;x,y,z\in X.
\end{gather*}
We have 
%\todo{$f(x):=d(x,0_{X})$}
$f(x):=d(x,0_X)=\|x\|^{\frac{1}{5}}$ for all $x\in X.$
Let $\sigma:=4$.
It is clear that $f$ is not $2$-homogeneous.
Let $x,y\in X$.
Then $\|x\pm y\|\le 2(\|x\|+\|y\|)$.
By the first inequality in Lemma~\ref{L1.5}\eqref{L1.5-1} with $\al:=\frac{4}{5}$ and $n=2$,
\begin{gather*}
\|x+y\|^{\frac{4}{5}}+\|x-y\|^{\frac{4}{5}}
\le 2(\|x\|+\|y\|)^{\frac{4}{5}}
\le 2(\|x\|^{\frac{4}{5}}+\|y\|^{\frac{4}{5}}).
\end{gather*}
By \eqref{G} and \eqref{C1},
\begin{gather*}
C_{\text{NJ}}^{(4)}(d)
=\sup_{(x,y)\ne (0_X,0_X)}\dfrac{f^4(x+y)+f^4(x-y)}{2^{3}(f^4(x)+f^4(y))}\le\frac{1}{4}.
\end{gather*}	
By Theorem~\ref{T2.5}\eqref{T2.5-1}, we have
$C_{\text{NJ}}^{(4)}(d)=\frac{1}{4}$.
Similarly to Example~\ref{Exam2.3}, it is straightforward to verify that the addition and scalar multiplication operations are continuous.
\end{example}

The next two examples show that the  assumption in Theorem~\ref{T2.5}\eqref{T2.5-3} is essential.

\begin{example}[the subadditivity is essential]\label{Exam2.2}
Let $X$ be a normed space, and
$d(x,y):=\|x-y\|+\left|\|x\|^2-\|y\|^2\right|$ for all $x,y\in X$.
We have $d$ is a metric on $X$, and
$f(x):=d(x,0_X)=\|x\|+\|x\|^2$ for all $x\in X$.
It is clear that $f$ is even. 
However,  $f$ is not subadditive since $f(2\mathbf{e})=6$ and $f(\mathbf{e})=2$, where $\mathbf{e}\in X$ is a unit vector.
Let $\sigma\in [2,\infty)$.
By \eqref{G} and \eqref{C1},
\begin{gather*}
C_{\text{NJ}}^{(\sigma)}(d)
\ge\dfrac{f^\sigma(2\mathbf{e})+f^\sigma(0_X)}{2^{\sigma-1}(f^\sigma(\mathbf{e})+f^\sigma(\mathbf{e}))}=(3/2)^\sigma>2.
\end{gather*}
Similarly to Example~\ref{Exam2.3}, it is straightforward to verify that the addition and scalar multiplication operations are continuous.
\end{example}

\begin{example}[the evenness is essential]\label{Exam2.4}
Let $X$ be a Hilbert space, $\mathbf{e}\in X$ be a unit vector, and
\begin{gather*}
h(x):=\begin{cases}
\|x\|+\langle x,\mathbf{e}\rangle &\text{if } \langle x,\mathbf{e}\rangle \ge 0,\\
\|x\|-4\langle x,\mathbf{e}\rangle &\text{if } \langle x, \mathbf{e}\rangle<0
\end{cases}
\end{gather*}
for all $x\in X$.
Define 
\begin{gather*}
d(x,y):=\begin{cases}
h(x)+h(y) &\text{if } x\ne y,\\
0 &\text{if } x=y
\end{cases}
\end{gather*}
for all $x,y\in X$.
Observe that $d$ is a metric on $X$, and
$f(x):=d(x,0_X)=h(x)$ for all $x\in X$.
It is clear that $f$ is subadditive.
However, it is not even since $f(-\mathbf{e})=5\ne 2=f(\mathbf{e})$.
Let $\sigma\in[1,\infty)$.
By \eqref{G} and \eqref{C1},
\begin{align*}
C_{\text{NJ}}^{(\sigma)}(d)\ge
\dfrac{f^\sigma((1+k)\cdot \mathbf{e})+f^\sigma((1-k)\cdot \mathbf{e})}{2^{\sigma-1}(f^\sigma(\mathbf{e})+f^\sigma(k\cdot \mathbf{e}))}
=\frac{(2k+2)^\sigma+(5k-5)^\sigma}{2^{\sigma-1}(2^\sigma+(2k)^\sigma)}:=\gamma_k
\end{align*}	
for all $k\in\N$. 
We have $\lim_{k\to\infty}\gamma_k=\frac{2^\sigma+5^\sigma}{2^{2\sigma-1}}>2.$
Thus, $C_{\text{NJ}}^{(\sigma)}(d)> 2$.
Let $x_k:=\mathbf{e}$ and $y_k:=\frac{1}{k}\mathbf{e}$ $(k\in\N)$. 
Then $d(y_k,0_X)\to 0$ as $k\to\infty$, and
\begin{align*}
d(x_k+y_k,\mathbf{e}+0_X)=d\left(\left(1+\frac{1}{k}\right)\mathbf{e},\mathbf{e}\right)
=h\left(\left(1+\frac{1}{k}\right)\mathbf{e}\right)+h(\mathbf{e})
=2\left(1+\frac{1}{k}\right)+2\to 4
\end{align*}
as $k\to\infty$.
Moreover, let $z_k:=\mathbf{e}$ and $\lambda_k:=1+\frac{1}{k}$ $(k\in\N)$. 
Then  $\lambda_k\to 1$ as $k\to\infty$, and
\begin{align*}
d(\lambda_k z_k,1\cdot\mathbf{e})=d\left(\left(1+\frac{1}{k}\right)\mathbf{e},\mathbf{e}\right)\to 4\;\;\text{as}\;\; k\to\infty.
\end{align*}
Thus, the addition and scalar multiplication operations are not continuous.
\end{example}

\if{
\DHH{22/3/26. 
Examples of non-normable metrics that satisfy only one of the two conditions in Proposition~\ref{T2.5}\eqref{T2.5-3} yield $C_{\rm{NJ}}^{(\sigma)}(d)>2$; see Examples \ref{Exam2.2} and \ref{Exam2.4}. Example~\ref{Exam2.4} satisfies \eqref{T2.5-1} but not \eqref{T2.5-3}. Example~\ref{Exam2.3} satisfies \eqref{T2.5-3} but not \eqref{T2.5-1}. Example~\ref{Exam2.5} satisfies both \eqref{T2.5-1} and \eqref{T2.5-3}.
}
}\fi

\if{
The following staement shows that condition \eqref{T2.5-3}  holds for a broader
class of distances.
\begin{proposition}\label{P2.5}
If $f$ is convex, $2$-homogeneous and even (i.e., $f(-x)=f(x)$ for all $x\in X$), then condition \eqref{T2.5-4} is satisfied.
\end{proposition}	
\begin{proof}
By the assumptions,
\begin{gather*}
f(x+y)=f\left(\frac{2x+2y}{2}\right)\leq \frac{f(2x)+f(2y)}{2}
=f(x)+f(y),\\
f(x-y)=f\left(\frac{2x-2y}{2}\right)\leq \frac{f(2x)+f(-2y)}{2}
=\frac{f(2x)+f(2y)}{2}
=f(x)+f(y)
\end{gather*}
for all $x,y\in X$.
Thus, condition \eqref{T2.5-4} is satisfied.
\end{proof}	
}\fi

\if{
\DHH{9/3/26.
An example of a function $f$ does not satisfy condition \eqref{T2.5-4} yields $C_{\rm{NJ}}^{(\sigma)}(d)>2$; see Example~\ref{Exam2.2}. An example of a function $f$ satisfying \eqref{T2.5-4} yields $C_{\rm{NJ}}^{(\sigma)}(d)\le2$; see Examples~\ref{Exam2.3} and \ref{Exam2.5}.
}
}\fi
\if{
\begin{remark}
\begin{enumerate}[\rm (i)]
\item	
Each of the Examples \ref{E3.3}-\ref{Exam2.5} provides the corresponding non-normable metrics associated with the properties in Proposition~\ref{T2.5}.
The following table summarizes these properties in each example.
\begin{center}
\begin{tabular}{clccc}
\toprule
& 
& 2-homogeneous
& subadditive
& even
\\
\midrule
& Example \ref{E3.3}
			& $\times$ & $\times$ & \checkmark \\
			& Example \ref{Exam2.2}
			& $\times$ & $\times$ & \checkmark \\
			& Example \ref{Exam2.4}
			& \checkmark & \checkmark & $\times$  \\
			& Example \ref{Exam2.3}
			& $\times$ & \checkmark & \checkmark \\
			&  Example \ref{Exam2.5}
			& \checkmark & \checkmark & \checkmark  \\
			\bottomrule
		\end{tabular}
	\end{center}
\item
In  Examples~\ref{E3.3} and \ref{Exam2.2}, the $2$-homogeneity condition is not satisfied. 
However, in the first example the constant is less than $1$, while in the second example it is greater than $1$. 
On the other hand, Examples~\ref{Exam2.4} and \ref{Exam2.5} can be viewed as illustrations of Proposition~\ref{T2.5}\eqref{T2.5-1}: the $2$-homogeneity condition is satisfied and the constant is greater than $1$. 
These observations show that the $2$-homogeneity property is far from being necessary for the constant to be greater than $1$.
\end{enumerate}
\end{remark}
}\fi

\if{
\begin{remark}
\begin{enumerate}[\rm (i)]
\item\label{R2.3-1}
Condition \eqref{A1} is equivalent to $f(2^kx)=2^kf(x)$ for all $x\in X$ and $k\in\Z$.
Indeed, if $k$ is non-negative, the statement is obvious.
If $k$ is negative, then
$$f(2^kx)=2^k\cdot2^{-k}f(2^kx)=2^kf(2^{-k}2^kx)=2^kf(x).$$
\item
The condition $\sigma\in[1,2]$ in part \eqref{P2.2-2} of Proposition~\ref{P2.2} is to ensure that $2^{\frac{2-\sigma}{2}}\ge 1$.
\item\label{R2.3-2}
We have $\mathcal{G}^{(\sigma)}_{d}(x,0_X)=2^{2-\sigma}$ for all $x\in X$ with $x\ne 0_X$.
By \eqref{C1}, $C_{\text{NJ}}^{(\sigma)}(d)\ge 2^{2-\sigma}$.
\end{enumerate}
\end{remark}
}\fi

\if{
The following  examples show that the assumptions in Theorem~\ref{T2.5}\eqref{T2.5-2} and \eqref{T2.5-3} are essential and cannot be dropped.
\begin{example}[the $2$-homogeneity is essential]
Let $\R^2$ be equipped with the Euclidean norm $\|\cdot\|$, and $\varphi:\R\to\R^2$ be given by $\varphi(x):=(x,x^2)$ for all $x\in\R$.
Consider the metric $d(x,x'):=\|\varphi(x)-\varphi(x')\|$ for all $x,x'\in\R$. 
We have $f(x):=d(x,0)=|x|\sqrt{1+x^2}$ for all $x\in\R$.
Observe that $f$ is convex and even.
For $x:=1$, we have $f(2x)=2\sqrt 5$ and $f(x)=\sqrt{2}$.
Thus, $f$ is not satisfied $2$-homogeneous. 
 Let $\sigma\in(1,\infty)$. 
Then
\begin{align*}
C_{\text{NJ}}^{(\sigma)}(d)
&\overset{\eqref{C1}}{\ge}\dfrac{f^\sigma(2n)+f^\sigma(0)}{2^{\sigma-1}(f^\sigma(n)+f^\sigma(n))}=\left(\frac{1+4n^2}{1+n^2}\right)^{\sigma/2}\to 2^\sigma\;\;\text{as}\;\;n\to\infty.
\end{align*}	
Thus, $C_{\text{NJ}}^{(\sigma)}(d)\geq 2^\sigma>2$.
\end{example}
}\fi
\if{
\begin{example}[the evenness is essential]
Let $g:\R\to\R$ be given by $g(x):=x$ if $x\ge 0$, and 
$g(x):=-4x$ if $x<0$.
Consider the metric 
\begin{gather*}
d(x,x'):=\begin{cases}
g(x)+g(x') &\text{if } x\ne x',\\
0 &\text{if } x=x'.
\end{cases}
\end{gather*}
We have $f(x):=d(x,0)=g(x)$ for all $x\in\R$.
Observe that $f$ is convex and $2$-homogeneous.
However, $f$ is not even since $f(1)=1\ne 4=f(-1)$.
Let $\sigma\in[1,\infty)$.
Then
\begin{gather*}
C_{\text{NJ}}^{(\sigma)}(d)
\overset{\eqref{C1}}{\ge}
\dfrac{f^\sigma(3)+f^\sigma(-1)}{2^{\sigma-1}(f^\sigma(1)+f^\sigma(2))}
=2\cdot\dfrac{3^\sigma+4^\sigma}{2^\sigma+4^\sigma}>2.
\end{gather*}	
\end{example}
}\fi

\if{
The follwing example shows that there exists a non-normable metric satisfies all the condition in Proposition~\ref{P2.5}.
\begin{example}
Let $\R^n$ be a vector space over the field of rational numbers $\mathbb{Q}$, $\mathbf{e}$ be a unit vector in $\R^n$ and $S:=\{\mathbf{e}, \sqrt{2}\mathbf{e}\}$.
Note that $S$ is linearly independent. 
By \cite[Theorem 4.2.1]{Mar09}, there exists a Hamel basis $\mathcal{H}$ of $\R^n$ such that $\mathcal{H}\supset S$. 
Define $g:\mathcal{H}\to\R$  by
\begin{gather*}
g(x):=\begin{cases}
0  & \text{\rm if } x\in\mathcal{H}\setminus\{\mathbf{e}\},\\
1 & \text{\rm if } x= \mathbf{e}.
\end{cases} 
\end{gather*}
By \cite[Theorem 5.2.2]{Mar09}, there exists a unique additive function $h:\R^n\to \R$ such that $h(x)=g(x)$ for all $x\in\mathcal{H}$.
Consider the metric $d(x,x'):=\|x-x'\|+\big|h(x-x')\big|$ for all $x,x'\in\R^n$, where $\|\cdot\|$ be a norm on $\R^n$.
Then $f(x):=d(x,0)=\|x\|+|h(x)|$ for all $x\in\R^n$.
Observe that $f$ is convex, even, $2$-homogeneous.
We have
\begin{align*}
f(\sqrt{2}\mathbf{e})
&=\|\sqrt{2}\mathbf{e}\|+|h(\sqrt{2}\mathbf{e})|=\sqrt{2}+0=\sqrt{2},\\
\sqrt 2\cdot f(\mathbf{e})
&=\sqrt{2}\left(\|\mathbf{e}\|+|h(\mathbf{e})|\right)=\sqrt{2}(1+1)=2\sqrt{2}.
\end{align*}
Thus, $f(\sqrt{2}\mathbf{e}) \ne \sqrt 2\cdot f(\mathbf{e})$, and consequently, $f$ is not a norm on $\R^n$. 
We are going to show that $C_{\text{NJ}}^{(\sigma)}(d)=2$ for all $\sigma\in [1,\infty)$.
For any $x\in\R^n$, there exist an integer $m\in\N$, and
 $q_1,\ldots, q_m\in \mathbb{Q}$, $x_1,\ldots,x_m\in\mathcal{H}$ such that $	x=\sum_{i=1}^{m}q_ix_i$. 
By the rational homogeneity of $h$ (\cite[Theorem 5.2.1]{Mar09}), we have
\begin{equation*}
	h(x)=h\left(\sum_{i=1}^{m}q_ix_i\right)=\sum_{i=1}^{m}q_ih(x_i)=\sum_{i=1}^{m}q_ig(x_i)\in\mathbb{Q}.
\end{equation*}
By \cite[Corollary 5.2.2]{Mar09}, $h$ is discontinuous.
We now show that $\sup_{0<\|x\|\le 1}|h(x)|=\infty$. 
Suppose that $|h(x)|\le M$ for all $x\in\R^n$ with $0<\|x\|\le 1$ and some $M>0$. 
Let $x\ne 0$ and $q\in\mathbb{Q}$ such that $\frac{1}{2\|x\|}<q<\frac{1}{\|x\|}$. 
Then $|h(x)|=\frac{1}{q}|h(qx)|\le \frac{M}{q}<2M\|x\|\to 0$ as $x\to 0$, which contradicts the discontinuity of $h$. 
Consequently, there exists a sequence $\{x_k\}\subset\R^n$ such that $0<\|x_k\|\leq 1$ for all $k\in\N$ and $M_k:=|h(x_k)|\to\infty$ as $k\to\infty$.
Let $y_k:=-x_k+h(x_k)\cdot\mathbf{e}$. Then
\begin{equation}\label{E2.5-1}
	f(x_k+y_k)=f(h(x_k)\cdot\mathbf{e})
	=\|h(x_k)\cdot\mathbf{e}\|+|h(h(x_k)\cdot\mathbf{e})|
	=M_k+M_k=2M_k,
\end{equation}
\begin{align}
	f(x_k-y_k)&=f(2x_k-h(x_k)\cdot\mathbf{e})\nonumber\\
	&=\|2x_k-h(x_k)\cdot\mathbf{e}\|+|h(2x_k-h(x_k)\cdot\mathbf{e})|\nonumber\\
	&\ge \|h(x_k)\cdot\mathbf{e}\|-\|2x_k\|+|2h(x_k)-h(x_k)|\nonumber\\
	&\ge M_k-2+M_k=2M_k-2\label{E2.5-2},
\end{align}
\begin{equation}\label{E2.5-3}
	f(x_k)=\|x_k\|+|h(x_k)|\leq 1+M_k,
\end{equation}
\begin{align}
	f(y_k)&=f(-x_k+h(x_k)\cdot\mathbf{e})\nonumber\\
	&=\|-x_k+h(x_k)\cdot\mathbf{e}\|+|h(-x_k+h(x_k)\cdot\mathbf{e})|\nonumber\\
	&\le \|-x_k\|+\|h(x_k)\cdot\mathbf{e}\|+|-h(x_k)+h(x_k)|\le 1+M_k.\label{E2.5-4}
\end{align}
Let $\sigma\in[1,\infty)$. By \eqref{G}, \eqref{C1}, \eqref{E2.5-1}-\eqref{E2.5-4} and for sufficiently large $k$, we have
\begin{align*}
C_{\text{NJ}}^{(\sigma)}(d)&\geq\frac{f^\sigma(x_k+y_k)+f^\sigma(x_k-y_k)}{2^{\sigma-1}\left(f^\sigma(x_k)+f^\sigma(y_k)\right)}\\
&\ge\frac{(2M_k)^\sigma+(2M_k-2)^\sigma}{2^{\sigma-1}\left((1+M_k)^\sigma+(1+M_k)^\sigma\right)}\\
&=\frac{2^\sigma+\left(2-\frac{2}{M_k}\right)^\sigma}{2^\sigma\left(1+\frac{1}{M_k}\right)^\sigma}\longrightarrow2\quad\text{as } k\to\infty.
\end{align*}
Hence, $C_{\text{NJ}}^{(\sigma)}(d)\ge 2$. By Theorem~\ref{T2.5}\eqref{T2.5-3}, we have $C_{\text{NJ}}^{(\sigma)}(d)=2$ for all $\sigma\in[1,\infty)$.
\end{example}
}\fi

It is well known that the conventional von Neumann-Jordan constant of a norm equals one if and only if the norm is induced by an inner product \cite[Theorem~II]{JorNeu35}, or equivalently, the norm satisfies the parallelogram law \cite[Theorem~I]{JorNeu35}. 
The following theorem presents a metric version of this result.
\begin{theorem}\label{P2.2}
Let  $\sigma\in[1,\infty)$.
The following assertions hold.
\begin{enumerate}[\rm (i)]
\item\label{P2.2-1}
If 
\begin{gather}\label{P4.5-6}
	f^\sigma(x+y)+f^\sigma(x-y)= 2^{\frac{\sigma}{2}}(f^\sigma(x)+f^\sigma(y))
	\;\;\text{for all}\;\;x,y\in X,
\end{gather} 
then $C_{\text{\rm NJ}}^{(\sigma)}(d)=2^{1-\frac{\sigma}{2}}$.
\item\label{P2.2-2}
Suppose that  $f$ is 2-homogeneous.
If $C_{\text{\rm NJ}}^{(\sigma)}(d)=2^{1-\frac{\sigma}{2}}$, then condition \eqref{P4.5-6} is satisfied.
\end{enumerate}	
\end{theorem}	

\begin{proof}
\begin{enumerate}[\rm (i)]
\item		
Suppose that condition \eqref{P4.5-6} is satisfied.
Then
\begin{gather*}
\mathcal{G}^{(\sigma)}_{d}(x,y)\overset{\eqref{G}}{=}\dfrac{f^\sigma(x+y)+f^\sigma(x-y)}{2^{\sigma-1}(f^\sigma(x)+f^\sigma(y))}=2^{1-\frac{\sigma}{2}}
\end{gather*}	
for all $x,y\in X$ with $(x,y)\ne (0_X,0_X)$. 
By \eqref{C1}, $C_{\text{\rm NJ}}^{(\sigma)}(d)=2^{1-\frac{\sigma}{2}}$.
\item
Suppose that $f$ is 2-homogeneous and $C_{\text{\rm NJ}}^{(\sigma)}(d)=2^{1-\frac{\sigma}{2}}$.
By \eqref{G} and \eqref{C1}, 
\begin{gather}\label{P4.5-7}
f^\sigma(u+v)+f^\sigma(u-v)\le 2^{\frac{\sigma}{2}}(f^\sigma(u)+f^\sigma(v))
\end{gather}	
for all $u,v\in X$.
Let $x,y\in X$ with $(x,y)\ne (0_X,0_X)$. 
Then condition \eqref{P4.5-7} is satisfied with $x$ and $y$ in place of $u$ and $v$, respectively.
On the other hand, in view of  \eqref{P4.5-7} with $u:=\frac{x+y}{2}$ and $v:=\frac{x-y}{2}$, we have
\begin{align*}
f^\sigma(x)+f^\sigma(y)
\le 2^{\frac{\sigma}{2}}\left(f^\sigma\left(\frac{x+y}{2}\right)+f^\sigma\left(\frac{x-y}{2}\right)\right)
=\dfrac{f^\sigma(x+y)+f^\sigma(x-y)}{2^{\frac{\sigma}{2}}},
\end{align*}	
where the latter follows from Remark~\ref{R1.2}\eqref{R1.2-2}.
Thus, condition \eqref{P4.5-6} is satisfied.
\end{enumerate}
The proof is complete.
\end{proof}	

%\begin{remark}
%When $f$ is a norm on $X$ and $\sigma:=2$, Theorem~\ref{P2.2} recaptures 
%\end{remark}	

The next example shows that there exists a non-normable metric satisfying the statements in Theorem~\ref{P2.2}.
\begin{example}\label{E3.13}
Let $(\R^n,\R)$ be a vector space equipped with a norm $\|\cdot\|$, $\mathbf{e}\in\R^n$ be a unit vector,  and $S:=\{\mathbf{e},\sqrt{2}\mathbf{e}\}$.
Consider the restriction $(\R^n,\mathbb{Q})$ of $(\R^n,\R)$.
Observe that $S$ is a linearly independent set in $(\R^n,\mathbb{Q})$. 
By Lemma~\ref{L2.10}\eqref{L2.10-1}, there exists a  basis $\mathcal{B}\supset S$ of $(\R^n,\mathbb{Q})$.
For any $x\in\R^n$, define $A(x):=\sqrt{q_1^2+\cdots+q_l^2}$, where $l \in \mathbb{N}$ and $q_1,\ldots,q_l \in \mathbb{Q}$ are the coefficients in the representation of $x$ with respect to the basis $\mathcal{B}$.
Define 
\begin{gather}\label{E3.10-1}
d(x,y):=A(x-y)\;\;\text{for all}\;\; x,y\in\R^n. 
\end{gather}

It is clear that $d$ is nonnegative, symmetric, and $d(x,y)=0$ if and only if $x=y$.
Let $x,y,z \in \mathbb{R}^n$. Then there exist integers $m,k \in \mathbb{N}$ and finite sets $\mathbf{B} := \{\mathbf{b}_1,\ldots,\mathbf{b}_m\}$ and $\mathbf{C} := \{\mathbf{c}_1,\ldots,\mathbf{c}_k\}$
such that the vectors $x - y$ and $y - z$ can be represented in terms of elements of $\mathbf{B}$ and $\mathbf{C}$, respectively.
Let  $\mathbf{H}:=\mathbf{B}\cup\mathbf{C} \subset\mathcal{B}$.
Then, there exist  numbers $\al_i,\be_i\in\mathbb{Q}$ and vectors $\mathbf{h}_i\in\mathbf{H}$ $(i=1,\ldots,m+k)$ 
such that $x-y=\sum_{i=1}^{m+k}\alpha_i \mathbf{h}_i$ and $y-z=\sum_{i=1}^{m+k}\beta_i \mathbf{h}_i.$
Then $x-z=\sum_{i=1}^{m+k}(\alpha_i+\beta_i)\mathbf{h}_i$, and consequently,
\begin{align*}
d(x,z)
\overset{\eqref{E3.10-1}}{=}
\sqrt{\sum_{i=1}^{m+k}(\al_i+\be_i)^2}
\le \sqrt{\sum_{i=1}^{m+k}\al_i^2}+\sqrt{\sum_{i=1}^{m+k}\be^2_i}
\overset{\eqref{E3.10-1}}{=}d(x,y)+d(y,z).
\end{align*}	
Thus, $d$ is a metric on $\R^n$.

The function $f(x):=d(x,0_{\R^n})=A(x)$ for all $x\in \R^n$ is
 $\lambda$-homogeneous for any nonnegative $\lambda\in\mathbb{Q}$.
 Indeed, for $x\in\R^n$ and $\lambda\in\mathbb{Q}$ with $\lambda\ge 0$, there exist an integer $p\in\N$, scalars
$\gamma_1,\ldots, \gamma_p\in \mathbb{Q}$ and vectors $\mathbf{d}_1,\ldots,\mathbf{d}_p\in \mathcal{B}$ such that $x=\sum_{i=1}^{p}\gamma_i\mathbf{d}_i$.
Then
\begin{gather*}
	f(\lambda x)=\sqrt{\sum_{i=1}^{p}(\lambda \gamma_i)^2}=|\lambda|\sqrt{\sum_{i=1}^{p}\gamma^2_i}=\lambda f(x).
\end{gather*}
We have $\mathbf{e}=1\cdot\mathbf{e}$ and $\sqrt{2}\mathbf{e}=1\cdot\sqrt{2}\mathbf{e}$, and consequently, $f(\mathbf{e})=f(\sqrt{2}\mathbf{e})=1$.
Thus, $f$ is not a norm on $X$ since  $f(\sqrt{2}\mathbf{e})\ne \sqrt{2}f(\mathbf{e})$. 
Observe that condition \eqref{P4.5-6} is satisfied with $\sigma:=2$.
Indeed, for any $x,y\in\R^n$, there exists a finite subset of $\mathcal{B}$ such that $x,y,x+y,$ and $x-y$ are all represented using the same vectors in this set, and hence involve the same number of coefficients. 
Therefore, the parallelogram law follows from the classical Euclidean setting.
By Theorem~\ref{P2.2}, we have $C_{\text{NJ}}^{(2)}(d)=1$.

We now show that only the addition operation is continuous.
Observe that 
\begin{gather}\label{E3.10-2}
d(x+y,x'+y')\leq d(x,x')+d(y,y')\;\;\text{for all}\;\;x,y,x',y'\in\R^n.
\end{gather}
Indeed, for any $x,y,x',y'\in\R^n$, there exists a $s\in\N$,  $\theta_i,\vartheta_i, \theta_i',\vartheta_i'\in\mathbb{Q}$ and  $\mathbf{u}_i\in\mathcal{B}$ $(i=1,\ldots,s)$ 
such that $x=\sum_{i=1}^{s}\theta_i \mathbf{u}_i$, $y=\sum_{i=1}^{s}\vartheta_i \mathbf{u}_i$, $x'=\sum_{i=1}^{s}\theta_i'\mathbf{u}_i$, and $y'=\sum_{i=1}^{s}\vartheta_i'\mathbf{u}_i$. 
Then
% $x-x'=\sum_{i=1}^{s}(\alpha_i-\alpha_i')\mathbf{d}_i$, $y-y'=\sum_{i=1}^{s}(\beta_i-\beta_i')\mathbf{d}_i$, $x-x'+y-y'=\sum_{i=1}^{\ell}(\alpha_i-\alpha_i'+\beta_i-\beta_i')\mathbf{b}_i$, and consequently,
\begin{align*}
d(x+y,x'+y')
&\overset{\eqref{E3.10-1}}{=}\sqrt{\sum_{i=1}^{s}(\theta_i-\theta_i'+\vartheta_i-\vartheta_i')^2}\\
&\le \sqrt{\sum_{i=1}^{s}(\theta_i-\theta_i')^2}+\sqrt{\sum_{i=1}^{s}(\vartheta_i-\vartheta_i')^2}
\overset{\eqref{E3.10-1}}{=}d(x,x')+d(y,y').
\end{align*}
Let $x,y\in \R^n$ and $\{x_k\}, \{y_k\}\subset \R^n$ with $d(x_k,x)\to 0$ and $d(y_k,x)\to 0$ as $k\to\infty$. 
By \eqref{E3.10-2},
\begin{equation*}
d(x_k+y_k,x+y)\leq d(x_k,x)+d(y_k,y)\to 0 \;\;\text{as}\;\; k\to\infty.
\end{equation*}
On the other hand, let $z_k:=\mathbf{e}$ $(k\in\N)$ and
$\{\lambda_k\}\subset\mathbb{Q}$ with $\lambda_k\to\sqrt{2}$ as $k\to\infty$. 
Then 
\begin{equation*}
d(\lambda_k z_k,\sqrt{2}\mathbf{e})=A(\lambda_k\cdot\mathbf{e}+(-1)\cdot\sqrt{2}\mathbf{e})=\sqrt{\lambda_k^2+1}\to\sqrt{3}\;\;\text{as}\;\; k\to\infty.
\end{equation*}
Therefore, the scalar multiplication operation is not continuous.
\end{example}

The following example shows that the 2-homogeneity is essential for Theorem~\ref{P2.2}\eqref{P2.2-2}.
\begin{example}
Let $X$ be a Hilbert space, 
$d(x,y):=\min\{\|x-y\|,1\}$ for all $x,y\in X$, and $\sigma:=2$.
The function $f(x):=d(x,0_X)=\min\{\|x\|,1\}$ for all $x\in X$ is not 2-homogeneous. 
By Example~\ref{Exam2.3}, $C_{\text{NJ}}^{(2)}(d)=1$.
Let $\mathbf{e}\in X$ be a unit vector, and  $x=y:=\mathbf{e}$.
Then $f(x+y)=f(x)=f(y)=1$ and $f(x-y)=0$, and consequently,
$f^2(x+y)+f^2(x-y)\ne 2(f^2(x)+f^2(y))$.
\end{example}	

The following statement provides a necessary and sufficient condition under which a metric satisfying the parallelogram law with $\sigma:=2$ is normable.
\begin{theorem}\label{T3.14}
The following assertions hold.
\begin{enumerate}[\rm (i)]
\item\label{T3.14-1}
If $f$ is a norm induced by an inner product, then
\begin{gather}\label{C3.13-1}
f^2(x+y)+f^2(x-y)= 2(f^2(x)+f^2(y))\;\;\text{for all}\;\;x,y\in X.
\end{gather}	
\item\label{T3.14-2}
Suppose that $f$ is positively homogeneous.	
If condition \eqref{C3.13-1} is satisfied, then $f$ is a norm induced by an inner product.
\end{enumerate}		
\end{theorem}	

\begin{proof}
\begin{enumerate}[\rm (i)]
\item 
Suppose that $f$ is a norm induced by an inner product $\langle \cdot,\cdot\rangle$, i.e, $f(x)=\sqrt{\langle x,x\rangle}$ for all $x\in X$.
Let $x,y\in X$.
Then
\begin{gather*}
f^2(x+y)=\langle x,x\rangle+\langle y,y\rangle+2\langle x,y\rangle,\;\;
f^2(x-y)=\langle x,x\rangle+\langle y,y\rangle-2\langle x,y\rangle,
\end{gather*}	
and consequently, $f^2(x+y)+f^2(x-y)=\langle x,x\rangle+\langle y,y\rangle=f^2(x)+f^2(y).$
\item
Suppose that $f$ is positively homogeneous and condition \eqref{C3.13-1} is satisfied.
Define
\begin{gather}\label{C3.13-2}
\langle x,y\rangle:=\frac{1}{4}(f^2(x+y)-f^2(x-y))\;\;\text{for all}\;\;x,y\in X.
\end{gather}
We are going to show that $\langle\cdot,\cdot\rangle$
is an inner product on $X$.
We have $\langle x,x\rangle=f^2(x)\ge 0$ for all $x\in X$, and the equality holds if and only if  $x=0_X$.
By \eqref{C3.13-1}, $f^2(y)+f^2(-y)=2f^2(y)$, i.e., $f^2(-y)=f^2(y)$ for all $y\in X$.
From this and the nonnegativity, we have $f$ is even, and consequently, it is absolutely homogeneous.
Then
\begin{gather*}
\langle y, x\rangle=\frac{1}{4}(f^2(y+x)-f^2(y-x))=\frac{1}{4}(f^2(x+y)-f^2(x-y))=\langle x,y\rangle,\\
\langle -x,y\rangle=\frac{1}{4}(f^2(y-x)-f^2(y+x))=-\frac{1}{4}(f^2(x+y)-f^2(x-y))=-\langle x,y\rangle
\end{gather*}	
for all $x,y\in X$.\\
\textbf{Claim 1.} $\langle x+y,z\rangle=\langle x,z\rangle+\langle y,z\rangle$ for all $x,y,z\in X$.
Indeed,
\begin{align*}
\langle x+y,z\rangle+\langle x-y,z\rangle
&=\frac{1}{4}\left[\left(f^2(x+z+y)+f_2^2(x+z-y)\right)-\left(f^2(x-z+y)+f^2(x-z-y)\right)\right]\\
&\overset{\eqref{C3.13-1}}{=}\frac{1}{4}\left[2(f^2(x+z)+f^2(y))-2(f^2(x-z)+f^2(y))\right]\\
&=\dfrac{1}{2}(f^2(x+z)-f^2(x-z))
\overset{\eqref{C3.13-2}}{=} 2\langle x,z\rangle.
\end{align*}
By interchanging $x$ and $y$, we obtain $\langle x+y,z\rangle+\langle y-x,z\rangle=2\langle y,z\rangle$.
Note that $\langle x-y,z\rangle+\langle y-x,z\rangle=0$.
Thus, $\langle x+y,z\rangle=\langle x,z\rangle+\langle y,z\rangle$.\\
\textbf{Claim 2.}  $\langle\lambda x,y\rangle=\lambda \langle x,y\rangle$ for all $x,y\in X$ and $\lambda\in\R$.
Let $x,y\in X$.
In view of {Claim 1}, the function $\lambda\mapsto \xi_{xy}(\lambda):=\langle\lambda x,y\rangle$ is additive.
 By Remark~\ref{R1.2}\eqref{R1.2-3}, $\xi_{xy}$ is $\lambda$-homogeneous for any $\lambda\in\mathbb{Q}$. 
We have $\xi_{xy}(0)=\langle 0_X,y\rangle=0$.
By the  absolute homogeneity, we have for any $0<|\gamma|\leq 1$ that
\begin{align*}
|\xi_{xy}(\gamma)|	
&=\frac{1}{4}\left|f^2(\gamma x+y)-f^2(\gamma x-y)\right|
\leq \frac{1}{4}(f^2(\gamma x+y)+f^2(\gamma x-y))\\
&\overset{\eqref{C3.13-1}}{=}\frac{1}{2}(f^2(\gamma x)+f^2(y))
=\frac{1}{2}(\gamma^2f^2(x)+f^2(y))\leq \frac{1}{2}(f^2(x)+f^2(y))=:M.
\end{align*}
For any $\lambda\ne 0$, there exists a $q\in\mathbb{Q}$ such that $\frac{1}{2|\lambda|}<q<\frac{1}{|\lambda|}$ since $\mathbb{Q}$ is dense in $\R$.
Then $0<|q\lambda|<1$, and consequently, $|\xi_{xy}(\lambda)|=\frac{1}{q}|\xi_{xy}(q\lambda)|\le \frac{M}{q}<2M|\lambda|\to 0$ as $\lambda\to 0$. 
Thus, $\xi_{xy}$ is continuous at $0$.
In view of the additivity, $\xi_{xy}$ is continuous on $\R$.
Let $\lambda\in\R$.
Since $\mathbb{Q}$ is dense in $\R$, one can find a sequence $\{q_k\}\subset\mathbb{Q}$ such that $q_k\to\lambda$ as $k\to\infty$.
In view of the continuity of $\xi_{xy}$, we have
\begin{gather*}
\xi_{xy}(\lambda)=\lim_{k\to\infty}\xi_{xy}(q_k)=\lim_{k\to\infty}q_k\xi_{xy}(1)=\langle x,y\rangle\lim_{k\to\infty}q_k=\lambda\langle x,y\rangle.
\end{gather*}

Therefore, the mapping $\langle\cdot,\cdot\rangle$ is an inner product on $X$. Observe that $f(x)\overset{\eqref{C3.13-2}}{=}\sqrt{\langle x, x\rangle}$ for all $x\in X$.
Thus, $f$ is a norm on $X$.
\end{enumerate}	
The proof is complete.
\end{proof}	

\begin{remark}
In \cite[Theorem~I]{JorNeu35}, Jordan and Neumann showed that a norm satisfying condition \eqref{C3.13-1} is induced by an inner product. 
Theorem~\ref{T3.14}\eqref{T3.14-2} strengthens this result by showing that the positive homogeneity of the function $f$ together with condition \eqref{C3.13-1} guarantees the norm and inner product structure of $f$.
\end{remark}

The following statement establishes parametrized formulas for the constant \eqref{C1}.
\begin{proposition}\label{P2.10}
Let $\sigma\in[1,\infty)$ and 
\begin{gather*}
H^\sigma_d(x,y,t):=\dfrac{f^\sigma(x+ty)+f^\sigma(x-ty)}{2^{\sigma-1}(1+t^\sigma)}\;\;
\text{for all}\;\;x,y\in X\;\;\text{and}\;\;t\in[0,1].
\end{gather*}		
If $f$ is absolutely homogeneous, then 
\begin{align}
C_{\text{\rm NJ}}^{(\sigma)}(d)&=\sup_{f(x)=f(y)=1,\;t\in[0,1]}H^\sigma_d(x,y,t)\label{C2}\\
&=\sup_{f(x)=1,\;f(y)\le1,\;t\in[0,1]}H^\sigma_d(x,y,t)\label{C3}\\
&=\sup_{f(x)\le 1,\;f(y)=1,\;t\in[0,1]}H^\sigma_d(x,y,t)\label{C4}\\
&=\sup_{f(x)\le1,\;f(y)\le1,\;t\in[0,1]}H^\sigma_d(x,y,t)\label{C5}.
\end{align}	
\end{proposition}

\begin{proof}
Let $A_1$, $A_2$, $A_3$ and $A_4$ denote the right-hand sides of \eqref{C2}-\eqref{C5}, respectively.
It is clear that $A_1\le A_2$ and $A_3\le A_4$.
Since $H^\sigma_d(x,y,0)=2^{2-\sigma}$ for all  $x,y\in X$ with $f(x)=f(y)=1$, we have
$A_i\ge 2^{2-\sigma}$ $(i=1,2,3,4)$. \\
\textbf{Claim 1.} $C_{\text{NJ}}^{(\sigma)}(d)\le A_1$.
Let $x,y\in X$ with $(x,y)\ne (0_X, 0_X)$.
By \eqref{G}, we have $\mathcal{G}_{d}^{(\sigma)}(x,0_X)=\mathcal{G}_{d}^{(\sigma)}(0_X,y)=2^{2-\sigma}\le A_1$.
Suppose that $x\ne 0_X$ and $y\ne 0_X$.
Let $u:=\frac{x}{f(x)}$, $v:=\frac{y}{f(y)}$ and $t:=\frac{f(y)}{f(x)}$.
Then $f(u)=f(v)=1$.
If $f(x)\ge f(y)$, then $t\in(0,1]$, and consequently,
\begin{align*}
\mathcal{G}_{d}^{(\sigma)}(x,y)
&\overset{\eqref{G}}{=}\dfrac{f^\sigma(f(x)\cdot(u+tv))+f^\sigma(f(x)\cdot(u-tv))}{2^{\sigma-1}(f^\sigma(f(x)\cdot u)+f^\sigma(f(x)t\cdot v))}\\
&=\dfrac{f^\sigma(u+tv)+f^\sigma(u-tv)}{2^{\sigma-1}(1+t^\sigma)}=H^\sigma_d(u,v,t)\le A_1.
\end{align*}
If $f(x)<f(y)$, then $t':=t\iv\in(0,1)$ and
\begin{align*}
\mathcal{G}_{d}^{(\sigma)}(x,y)
&\overset{\eqref{G}}{=}\dfrac{f^\sigma(f(y)\cdot(v+t'u))+f^\sigma(f(y)\cdot(v-t'u))}{2^{\sigma-1}(f^\sigma(f(y)\cdot v)+f^\sigma(f(y)t'\cdot u))}\\
&=\dfrac{f^\sigma(v+t'u)+f^\sigma(v-t'u)}{2^{\sigma-1}(1+t'^\sigma)}=H^\sigma_d(v,u,t')\le A_1.
\end{align*}
In view of \eqref{C1}, we have $C_{\text{NJ}}^{(\sigma)}(d)\le A_1$.\\
\textbf{Claim 2.} $A_2\le A_3$. 
Let $x,y\in X$ with $f(x)=1$, $f(y)\le1$, and $t\in [0,1]$. 
If $y=0_X$, then
\begin{equation*}
H^\sigma_d(x,y,t)=\dfrac{f^\sigma(x)+f^\sigma(x)}{2^{\sigma-1}(1+t^\sigma)}=\frac{2}{2^{\sigma-1}(1+t^\sigma)}\le \frac{2}{2^{\sigma-1}}=2^{2-\sigma}\le A_3.
\end{equation*}
Suppose that $y\ne 0_X$. 
Let $z:=\frac{y}{f(y)}$ and $t':=tf(y)\in[0,1]$. 
Then $f(z)=1$, and
\begin{equation*}
H^\sigma_d(x,y,t)=\dfrac{f^\sigma(x+t'z)+f^\sigma(x-t'z)}{2^{\sigma-1}(1+t^\sigma)}
\le \dfrac{f^\sigma(x+t'z)+f^\sigma(x-t'z)}{2^{\sigma-1}(1+t'^\sigma)}=H^\sigma_d(x,z,t')\le A_3.
\end{equation*}
\textbf{Claim 3.} $A_4\le C_{\text{\rm NJ}}^{(\sigma)}(d)$. Let $x,y\in X$ with $f(x)\le 1, f(y)\le1$, and $t\in [0,1]$. If $x=0_X$, then
\begin{equation*}
H^\sigma_d(x,y,t)=\dfrac{f^\sigma(ty)+f^\sigma(-ty)}{2^{\sigma-1}(1+t^\sigma)}=\frac{2t^\sigma f^\sigma(y)}{2^{\sigma-1}(1+t^\sigma)}\le \frac{2}{2^{\sigma-1}}\le C_{\text{\rm NJ}}^{(\sigma)}(d),
\end{equation*}
where the last inequality comes from Theorem~\ref{T2.5}\eqref{T2.5-1}.
Suppose that $x\ne 0_X$.
Let $w:=ty$.
Then $f^\sigma(w)=t^\sigma f^\sigma(y)\leq t^\sigma$, and consequently,
\begin{equation*}
H^\sigma_d(x,y,t)=\dfrac{f^\sigma(x+ty)+f^\sigma(x-ty)}{2^{\sigma-1}(1+t^\sigma)}
\le \dfrac{f^\sigma(x+w)+f^\sigma(x-w)}{2^{\sigma-1}(f^\sigma(x)+f^\sigma(w))}
\overset{\eqref{C1}}{\le}
 C_{\text{\rm NJ}}^{(\sigma)}(d).
\end{equation*}
This completes the proof.
\end{proof}

\begin{remark}
When $f$ is a norm, equality~\eqref{C2} can be found in~\cite[p.~2]{CuiHuaHudKac15}. 
If $f$ is a norm and $\sigma = 2$, the equalities are studied in~\cite[Formula~(1) and Proposition~2.2]{YanWan06}.

%{\color{red}
%When $\sigma\in[1,\infty)$ and $f$ is a norm, the equality in \eqref{C2} can be found in  \cite[p. 2]{CuiHuaHudKac15}, while the other equivalence presentations seem to be new. 
%When $\sigma=2$ and $f$ is a norm, the equalities in \eqref{C2}, \eqref{C3}, and \eqref{C5} can be found in
\end{remark}

\section{A von Neumann-Jordan constant 
of metrics on product spaces}\label{S4}

We first present a general framework for constructing metrics on product spaces.
Let $n\ge 2$ and $\Omega_n:=\{(t_1,\ldots,t_{n})\in\R^{n}_+\mid \sum_{i=1}^nt_i =1\}$.
Note that $\Omega_n$  is a convex and compact set in $\R^n$.
Let $\mathbf{e}_1,\ldots,\mathbf{e}_n$ be the standard basis vectors of $\R^n$.
Denote by $\pmb{\Psi}_n$ the class of all continuous convex functions $\psi:\Omega_n\to\R_+$ satisfying
\begin{gather}\label{b}
\psi(\mathbf{e}_1)=\cdots=\psi(\mathbf{e}_n)=1,
\end{gather}	
and
\begin{gather}\label{I}
	\psi(t_1,\ldots,t_n)\ge (1-t_i)\cdot\psi\left(\dfrac{t_1}{1-t_i},\ldots,\dfrac{t_{i-1}}{1-t_i},0,\dfrac{t_{i+1}}{1-t_i},\ldots,\dfrac{t_{n}}{1-t_i}\right)
\end{gather}	
for all $(t_1,\ldots,t_n)\in\Omega_n$ with $t_i<1$ 
$(i=1,\ldots,n)$.

\begin{remark}
The class $\pmb{\Psi}_2$ was first introduced by Bonsall and Duncan \cite[p.~37]{BonDun73} for the study of absolute norms on $\mathbb{C}^2$. 
Saito et al. \cite{SaiKatTak00} subsequently extended this construction to $\mathbb{C}^n$ $(n \ge 2)$. 
Recently, the class $\pmb{\Psi}_n$ has been employed to construct general norms on a product of vector spaces~\cite{Cuo25}.
\end{remark}

The following metric construction was recently established in \cite{HieTamCuo26}.
\begin{theorem}\label{T3.10}
Let $X_i$ be a nonempty set, $d_i:X_i\times X_i\to\R_+$ be a function 
$(i=1,\ldots,n)$, and $\psi\in\pmb{\Psi}_n$.
Define
\begin{equation}\label{T2.2-1}
d_\psi(x,y)
:= \begin{cases}
\left(\sum_{i=1}^n d_i(x_i,y_i)\right)\cdot\psi\left(\dfrac{d_1(x_1,y_1)}{\sum_{i=1}^nd_i(x_i,y_i)},\ldots,\dfrac{d_n(x_n,y_n)}{\sum_{i=1}^nd_i(x_i,y_i)}\right)  & \text{\rm if } x\ne y,\\
0 & \text{\rm if } x= y
\end{cases} 
\end{equation}
for all $x:=(x_1,\ldots,x_n), y:=(y_1,\ldots,y_n)\in \mathrm{P}:=X_1\times\cdots\times X_n$.
Then, $d_\psi$ is a metric on $\mathrm{P}$ if and only if $d_i$ is a metric on $X_i$ $(i=1,\ldots,n)$.
\end{theorem}	

\begin{remark}
For a scalar $p\in[1,\infty]$, the function
\begin{gather}\label{ppsi}
\psi_{p}(t)
:=\begin{cases}
\left(t_1^p+\cdots+t_n^p\right)^{\frac{1}{p}}  & \text{if } p\in[1,\infty),\\
\max\{t_1,\ldots,t_n\} & \text{if } p=\infty
\end{cases} 
\end{gather}
for all $t:=(t_1,\ldots,t_n)$ belongs to $\pmb{\Psi}_n$; see \cite{Cuo26} for a detailed proof.			
Suppose that $d_i$ is a metric on $X_i$ $(i=1,\ldots,n)$.
By \eqref{T2.2-1}, the function
\begin{gather}\label{p-metric}
d_{\psi_p}(x,y)
= \begin{cases}
(d^p_1(x_1,y_1)+\cdots+d^p_n(x_n,y_n))^{\frac{1}{p}} & \text{if } p\in[1,\infty),\\
\max\{d_1(x_1,y_1),\ldots,d_n(x_n,y_n)\} & \text{if } p=\infty
\end{cases} 
\end{gather}	
for all $x:=(x_1,\ldots,x_n), y:=(y_1,\ldots,y_n)\in \mathrm{P}$ is a metric on $\mathrm{P}$.
Further examples of elements of $\pmb{\Psi}_n$ can be found in \cite{SaiKatTak00,SaiKatTak00+,TanOhwSai14} and the references therein.
\end{remark}	

The following result is needed for the subsequent analysis. 
It is a consequence of \eqref{I}. 
The reader is referred to \cite{Cuo25} for a detailed proof.
\begin{lemma}\label{L4.7}
Let  $\psi\in\pmb{\Psi}_n$.
The following assertions hold:
\begin{enumerate}[\rm (i)]
\item\label{L4.7-2}
if $0\le \al_i\le\be_i$ $(i=1,\ldots,n)$ with $\al:=\sum_{i=1}^{n}\al_i>0$ and  $\be:=\sum_{i=1}^{n}\be_i>0$, then
\begin{gather*}
\al\psi\left(\dfrac{\al_1}{\al},\ldots,\dfrac{\al_{n}}{\al}\right)\le \be \psi\left(\dfrac{\be_1}{\be},\ldots,\dfrac{\be_{n}}{\be}\right);
\end{gather*}
\item\label{L4.7-1}
$\max\{t_1,\ldots,t_n\}\le\psi(t)\le 1$ for all $t:=(t_1,\ldots,t_n)\in\Omega_n$.
\end{enumerate}	
\end{lemma}	

From now on, we assume that $d_i$ is a metric on a vector space $X_i$ $(i=1,\ldots,n)$, $d_\psi$ is given by \eqref{T2.2-1} for some $\psi\in \pmb{\Psi}_n$,
$f_i(u)
:=d_i(u,0_{X_i})$ for all $u\in X_i$ $(i=1,\ldots,n)$, and 
$f_\psi(x):=d_\psi(x,0_{\mathrm{P}})$ for all $x\in \mathrm{P}:=X_1\times\cdots\times X_n$.
Under these setting, we have
\begin{gather}\label{C4.5-3}
f_\psi(x)
= \begin{cases}
\left(\sum_{i=1}^n f_i(x_i)\right)\cdot\psi\left(\dfrac{f_1(x_1)}{\sum_{i=1}^nf_i(x_i)},\ldots,\dfrac{f_n(x_n)}{\sum_{i=1}^nf_i(x_i)}\right)  & \text{\rm if } x\ne 0_\mathrm{P},\\
0 & \text{\rm if } x=0_\mathrm{P}
\end{cases} 
\end{gather}	
for all $x:=(x_1,\ldots,x_n)\in \mathrm{P}$.

\subsection{The constant of general metrics}

The next statement is a product  version of Theorem~\ref{T2.5}.
\begin{corollary}\label{C4.5}
Let $\sigma\in[1,\infty)$ and $\psi\in \pmb{\Psi}_n$.
The following assertions hold true:
\begin{enumerate}[\rm (i)]
\item\label{C4.5-0}
$C_{\text{\rm NJ}}^{(\sigma)}(d_\psi)\ge 2^{2-\sigma}$;
\item\label{C4.5-1}
if $f_1,\ldots,f_n$ are 2-homogeneous, then $C_{\text{\rm NJ}}^{(\sigma)}(d_\psi)\ge 1$;
\item\label{C4.5-2}
if $f_1,\ldots,f_n$ are subadditive and even, then
$C_{\rm{NJ}}^{(\sigma)}(d_\psi)\le 2$.
\end{enumerate}		
\end{corollary}	

\begin{proof}
Assertion \eqref{C4.5-0}  is a consequence of Theorem~\ref{T2.5}\eqref{T2.5-1}.	
In view of \eqref{C4.5-3}, if $f_1,\ldots,f_n$ are 2-homogeneous and even, then $f_\psi$ is also  2-homogeneous and even, respectively.
Suppose that $f_1,\ldots,f_n$ are subadditive.
Let $x:=(x_1,\ldots,x_n),y:=(y_1,\ldots,y_n)\in \mathrm{P}$.	
If $x+y= 0_\mathrm{P}$, then the inequality $f_\psi(x+y)\le f_\psi(x)+	f_\psi(y)$ is trivially satisfied.
Let $x+y\ne 0_\mathrm{P}$.
Define 
\begin{gather*}%\label{P4.4-2}
\al_i:=f_i(x_i),\;\;\be_i:=f_i(y_i)\;\;(i=1,\ldots,n),\;\;
\al:=\sum_{i=1}^{n}\al_i,\;\;\be:=\sum_{i=1}^{n}\be_i.
\end{gather*}		
Note that  $f_i(x_i+y_i)\le\al_i+\be_i$ $(i=1,\ldots,n)$.
By Lemma~\ref{L4.7}\eqref{L4.7-2},
\begin{eqnarray*}
f_\psi(x+y)	
&\overset{\eqref{C4.5-3}}{=}&\left(\sum_{i=1}^{n}f_i(x_i+y_i)\right)\cdot\psi\left(\dfrac{f_1(x_1+y_1)}{\sum_{i=1}^{n}f_i(x_i+y_i)},\ldots,\dfrac{f_n(x_n+y_n)}{\sum_{i=1}^{n}f_i(x_i+y_i)}\right)\\
&\le& (\al+\be)\psi\left(\dfrac{ \al_1+\be_1}{\al+\be},\ldots,\dfrac{\al_n+\be_n}{\al+\be}\right)\\
&=&(\al+\be)\psi\left(\dfrac{\al}{\al+\be}\cdot \left(\frac{\al_1}{\al},\ldots,\frac{\al_n}{\al}\right)+  \dfrac{\be}{\al+\be}\cdot \left(\frac{\be_1}{\be},\ldots,\frac{\be_n}{\be}\right)\right)\\
&\le& \al \psi\left(\frac{\al_1}{\al},\ldots,\frac{\al_n}{\al}\right)+ \be \psi\left(\frac{\be_1}{\be},\ldots,\frac{\be_n}{\be}\right)\;\;\;\;(\psi\;\; \text{is convex})\\
&\overset{\eqref{C4.5-3}}{=}& f_\psi(x)+ f_\psi(y).
\end{eqnarray*}	
Thus, $f_\psi$ is subadditive.
Therefore, assertions \eqref{C4.5-1} and \eqref{C4.5-2} are  consequences of
Theorem~\ref{T2.5}\eqref{T2.5-2} and \eqref{T2.5-3}, respectively.
\end{proof}	

The next theorem establishes quantitative relations between the constants of metrics generated by elements of $\pmb{\Psi}_n$.
\begin{theorem}\label{T4.7}
Let $\psi,\phi\in \pmb{\Psi}_n$, and $\sigma\in [1,\infty)$.
Then
\begin{gather*}
\left(\frac{\mathrm{m}_n}{\mathrm{M}_n}\right)^\sigma\cdot C_{\text{\rm NJ}}^{(\sigma)}(d_\phi) \le C_{\text{\rm NJ}}^{(\sigma)}(d_\psi)\le \left(\frac{\mathrm{M}_n}{\mathrm{m}_n}\right)^\sigma\cdot C_{\text{\rm NJ}}^{(\sigma)}(d_\phi),
\end{gather*}
where \begin{gather}\label{mM}
\mathrm{m}_n:=\min_{t\in\Omega_n}\dfrac{\psi(t)}{\phi(t)}
\;\;\text{and}\;\;
\mathrm{M}_n:=\max_{t\in\Omega_n}\dfrac{\psi(t)}{\phi(t)}.
\end{gather}	
\end{theorem}	

\begin{proof}
Let $x:=(x_1,\ldots,x_n)\in \mathrm{P}$ with $x\ne 0_\mathrm{P}$.
By \eqref{C4.5-3} and \eqref{mM},
\begin{align*}
f_\psi(x)
&= \left(\sum_{i=1}^{n}f_i(x_i)\right)\cdot\psi
\left(\dfrac{f_1(x_1)}{\sum_{i=1}^{n}f_i(x_i)},\ldots,\dfrac{f_n(x_n)}{\sum_{i=1}^{n}f_i(x_i)}\right)\\
&\le\mathrm{M}_n\cdot\left(\sum_{i=1}^{n}f_i(x_i)\right)\cdot\phi
\left(\dfrac{f_1(x_1)}{\sum_{i=1}^{n}f_i(x_i)},\ldots,\dfrac{f_n(x_n)}{\sum_{i=1}^{n}f_i(x_i)}\right)
=\mathrm{M}_n\cdot f_\phi(x).
\end{align*}	
Similarly, $f_\psi(x)\ge \mathrm{m}_n\cdot f_\phi(x)$.
Thus,
\begin{gather}\label{T2.4-1}
\mathrm{m}_n\cdot f_\phi(x)\le f_{\psi}(x)\le \mathrm{M}_n\cdot f_\phi(x)\;\;\text{for all}\;\;x\in \mathrm{P}.
\end{gather}	
Let $x,y\in \mathrm{P}$ with $(x,y)\ne (0_{\mathrm{P}},0_{\mathrm{P}})$.
Then,
\begin{eqnarray*}
f^\sigma_\phi(x+y)+f^\sigma_\phi(x-y)
&\overset{\eqref{T2.4-1}}{\le}& \frac{1}{\mathrm{m}_n^\sigma}\cdot\left(f^\sigma_\psi(x+y)+f^\sigma_\psi(x-y)\right)\\
&\overset{\eqref{C1}}{=}& \frac{1}{\mathrm{m}_n^\sigma}\cdot C_{\text{\rm NJ}}^{(\sigma)}(d_\psi)\cdot 2^{\sigma-1}\left(f^\sigma_\psi(x)+f^\sigma_\psi(y)\right)\\
&\overset{\eqref{T2.4-1}}{\le}& \frac{1}{\mathrm{m}_n^\sigma}\cdot \mathrm{M}_n^\sigma\cdot C_{\text{\rm NJ}}^{(\sigma)}(d_\psi)\cdot 2^{\sigma-1} \left(f^\sigma_\phi(x)+f^\sigma_\phi(y)\right),
\end{eqnarray*}	
and consequently,
\begin{gather*}
\mathcal{G}^{(\sigma)}_{d_\phi}(x,y)\overset{\eqref{G}}{=}\dfrac{f^\sigma_\phi(x+y)+f^\sigma_\phi(x-y)}{2^{\sigma-1}(f^\sigma_\phi(x)+f^\sigma_\phi(y))}\le
\left(\frac{\mathrm{M}_n}{\mathrm{m}_n}\right)^\sigma\cdot
C_{\text{\rm NJ}}^{(\sigma)}(d_\psi).
\end{gather*}
In view of this and \eqref{C1},  the second inequality is satisfied.
The proof for the first inequality is analogous.
\end{proof}

\begin{remark}
\begin{enumerate}[\rm (i)]
\item	
By Lemma~\ref{L4.7}\eqref{L4.7-1}, every function in $\pmb{\Psi}_n$ takes strictly positive values.
For any pair $\psi,\phi \in \pmb{\Psi}_n$, the continuity property ensures that the function $\frac{\psi}{\phi}$ attains its minimum and maximum on the compact set $\Omega_n\subset\R^n$.
\item
Normed versions of Theorem~\ref{T4.7} for the case $\sigma=2$ can be found in 
\cite{Cuo25},  \cite[Theorem~1]{SaiKatTak00}, and \cite[Theorem~5.2]{SaiKatTak00+}.
\end{enumerate}
\end{remark}

When $n=2$, the estimates in Theorem~\ref{T4.7} can be improved under  additional assumptions.
\begin{corollary}\label{C3.4}
Let $\psi,\phi\in \pmb{\Psi}_2$, $\sigma\in[1,\infty)$, $\mathrm{m}_2$ and $\mathrm{M}_2$ be given by \eqref{mM}.
Suppose that $C_{\text{\rm NJ}}^{(\sigma)}(d_\phi)=2^{1-\frac{\sigma}{2}}$, $f_1$ and $f_2$ are positively homogeneous, and one of them is even.
The following assertions hold:
\begin{enumerate}[\rm (i)]
\item\label{C3.4-1}
if $\psi(t)\ge\phi(t)$ for all $t\in\Omega_2$, then $C_{\text{\rm NJ}}^{(\sigma)}(d_\psi)= \mathrm{M}_2^\sigma\cdot2^{1-\frac{\sigma}{2}}$;
\item\label{C3.4-2}
if $\psi(t)\le\phi(t)$ for all $t\in\Omega_2$, then $C_{\text{\rm NJ}}^{(\sigma)}(d_\psi)=\mathrm{m}_2^{-\sigma}\cdot 2^{1-\frac{\sigma}{2}}$.
\end{enumerate}	
\end{corollary}

\begin{proof}	
We now prove assertion \eqref{C3.4-1}.
Suppose that  $\psi(t)\ge\phi(t)$ for all $t\in\Omega_2$.
Then $\mathrm{m}_2\ge 1$.
We now show that $C_{\text{\rm NJ}}^{(\sigma)}(d_\psi)\ge \mathrm{M}_2^\sigma\cdot2^{1-\frac{\sigma}{2}}$.
Let $x_i\in X_i$ with $x_i\ne 0_{X_i}$ $(i=1,2)$.
Define $u_i:=\frac{\xi}{f_i(x_i)}\cdot x_i$ $(i=1,2)$ for some $\xi>0$.
By the homogeneity of $f_1$ and $f_2$,
\begin{gather*}
f_i(u_i)=f_i\left(\frac{\xi}{f_i(x_i)}\cdot x_i\right)=\frac{\xi}{f_i(x_i)}\cdot f_i(x_i)=\xi\;\;(i=1,2).
\end{gather*}	
Let $t:=(t_1,t_2)\in\Omega_2$ satisfy $\mathrm{M}_2=\frac{\psi( t)}{\phi(t)}$.
Note that $t_1,t_2\ge 0$ and $t_1+t_2=1$.
Set
 $x:=(t_1u_1,0_{X_2})$ and
$y:=(0_{X_1},t_2 u_2)$.
Without loss of generality, we can suppose that $f_2$ is even.
We have $f_1(t_1u_1)=t_1f(u_1)=t_1\xi$ and $f_2(t_2u_2)=f_2(-t_2u_2)=t_2f(u_2)=t_2\xi$. 
By \eqref{b} and \eqref{C4.5-3},
\begin{gather}\label{C3.4-3}
f_\psi(x)=f_1( t_1 u_1)\cdot \psi\left(1,0 \right)=t_1\xi,\;\;
f_\psi(y)=f_2(t_2 u_2)\cdot \psi\left(0,1 \right)= t_2\xi,\\ 
\label{C3.4-4}
f_\psi(x+y)=f_\psi(x-y)=(t_1\xi+ t_2\xi)\cdot\psi\left(\frac{t_1\xi}{t_1\xi+t_2\xi},\frac{ t_2\xi}{t_1\xi+t_2\xi}\right)=
\xi \psi(t),
\end{gather}	
and similarly,
\begin{gather}\label{C3.4-8}
f_\phi(x)=t_1\xi,\;f_\phi(y)=t_2\xi,\;
f_\phi(x+y)=f_\phi(x-y)=\xi \phi(t).
\end{gather}	
By \eqref{C4.5-3}, $f_\phi$ are positively homogeneous.
By Theorem~\ref{P2.2}, 
\begin{gather}\label{C3.7-3}
f^\sigma_\phi(x+y)+f^\sigma_\phi(x-y)=2^{\frac{\sigma}{2}}\left(f^\sigma_\phi(x)+f^\sigma_\phi(y)\right).
\end{gather}	
By \eqref{C3.4-3}-\eqref{C3.7-3},
\begin{align*}
f^\sigma_\psi(x+y)+f^\sigma_\psi(x-y)
&=2\xi^\sigma\psi^\sigma(t)\\
&=2\xi^\sigma \mathrm{M}_2^\sigma\phi^\sigma(t)\\
&=\mathrm{M}_2^\sigma\left(f^\sigma_\phi(x+y)+f^\sigma_\phi(x-y)\right)\\
&=2^{\frac{\sigma}{2}}\mathrm{M}_2^\sigma\left(f^\sigma_\phi(x)+f^\sigma_\phi(y)\right)\\
&=2^{\frac{\sigma}{2}}\mathrm{M}_2^\sigma\left(f^\sigma_\psi(x)+f^\sigma_\psi(y)\right),
\end{align*}	
and consequently,
\begin{gather*}
\mathcal{G}^{(\sigma)}_{d_\psi}(x,y)\overset{\eqref{G}}{=} \frac{f^\sigma_\psi(x+y)+f^\sigma_\psi(x-y)}{2^{\sigma-1}(f^\sigma_\psi(x)+f^\sigma_\psi(y))}=\mathrm{M}_2^\sigma\cdot2^{1-\frac{\sigma}{2}}.
\end{gather*}
In view of \eqref{C1}, we have $C_{\text{\rm NJ}}^{(\sigma)}(d_\psi)\ge \mathrm{M}_2^\sigma\cdot2^{1-\frac{\sigma}{2}}$.
By Theorem~\ref{T4.7}, 
$C_{\text{\rm NJ}}^{(\sigma)}(d_\psi)\le \mathrm{M}_2^\sigma\cdot2^{1-\frac{\sigma}{2}}$.
This proves assertion \eqref{C3.4-1}.
The proof for assertion \eqref{C3.4-2} is analogous.
\end{proof}	

\begin{remark}
When $X_1=\cdots=X_n$ is a normed space and $\sigma=2$, Corollary~\ref{C3.4} recaptures the results  studied in \cite{Cuo25}, \cite[Theorem 1]{SaiKatTak00},  \cite[Remark~5.3]{SaiKatTak00+} and \cite[Theorems 2.3 \& 2.4]{Cao03}.
\end{remark}	

\begin{proposition}\label{P4.9}
Let  $\psi\in \pmb{\Psi}_n$ and  $\sigma\in[1,\infty)$.
If $C_{\text{\rm NJ}}^{(\sigma)}(d_\psi)=2^{1-\frac{\sigma}{2}}$ and $f_1,\ldots,f_n$ are 2-homogeneous, then $C_{\text{\rm NJ}}^{(\sigma)}(d_i)=2^{1-\frac{\sigma}{2}}$ $(i=1,\ldots,n)$.
\end{proposition}	

\begin{proof}
Suppose that $C_{\text{\rm NJ}}^{(\sigma)}(d_\psi)=2^{1-\frac{\sigma}{2}}$ and $f_1,\ldots,f_n$ are 2-homogeneous.
By \eqref{C4.5-3}, $f_\psi$ is  2-homogeneous.
Let $x,y\in X_1$.
Set $u:=(x,0_{X_2},\ldots,0_{X_n})$ and $v:=(y,0_{X_2},\ldots,0_{X_n})$. 
By Theorem~\ref{P2.2}\eqref{P2.2-2}, 
\begin{gather*}
f^\sigma_\psi(u+v)+f^\sigma_\psi(u-v)= 2^{\frac{\sigma}{2}}(f^\sigma_\psi(u)+f^\sigma_\psi(v)).
\end{gather*} 
By \eqref{b} and \eqref{C4.5-3},
\begin{gather*}
f_\psi(u)=f_1(x),\;\;f_\psi(v)=f_1(y),\;\;f_\psi(u+v)=f_1(x+y),\;\;f_\psi(u-v)=f_1(x-y).
\end{gather*}	
Thus, $f^\sigma_1(x+y)+f^\sigma_1(x-y)= 2^{\frac{\sigma}{2}}(f^\sigma_1(x)+f^\sigma_1(y))$.
By \eqref{G} and \eqref{C1}, we have $C_{\text{\rm NJ}}^{(\sigma)}(d_1)=2^{1-\frac{\sigma}{2}}$.
The proof for the other equalities are analogous.
\end{proof}

\begin{remark}
In view of Theorems~\ref{P2.2}\eqref{P2.2-2}, \ref{T3.14}\eqref{T3.14-2}	and Proposition~\ref{P4.9}, when $\sigma:=2$, the functions $f_1$ and $f_2$ in Corollary~\ref{C3.4} are norms on $X_1$ and $X_2$, respectively. 
\end{remark}	

\subsection{The constant of $p$-metrics}
In this section, we consider the particular case when $\psi:=\psi_p$ is given by \eqref{ppsi} with $p \in [1,\infty]$. 
The corresponding metric $d_{\psi_p}$ is defined by \eqref{p-metric}.
The $p$-version of \eqref{C4.5-3} has the form:
\begin{gather}\label{p}
f_{\psi_p}(x)
= \begin{cases}
(f^p_1(x_1)+\cdots+f^p_n(x_n))^{\frac{1}{p}} & \text{if } p\in[1,\infty),\\
\max\{f_1(x_1),\ldots,f_n(x_n)\} & \text{if } p=\infty
\end{cases} 
\end{gather}	
for all $x:=(x_1,\ldots,x_n)\in\mathrm{P}$.

The following statement provides an lower bound for the constant \eqref{C1} of $p$-metrics.

\begin{proposition}\label{T7.2}
Suppose that there exist distinct indices $i,j\in\{1,\ldots,n\}$ such that $f_i$ is positively homogeneous, $f_j$ is absolutely homogeneous, and $f_k$ is $2$-homogeneous for all $k\ne i,j$.
Then
\begin{gather}\label{T3.10-1}
C_{\text{\rm NJ}}^{(\sigma)}(d_{\psi_p})\ge
\begin{cases}	
2^{\frac{\sigma}{p}-\sigma+1}& \text{\rm if } p\in(1,2],\\
2^{1-\frac{\sigma}{p}}& \text{\rm if } p\in(2,\infty),\\
2& \text{\rm if } p\in\{1,\infty\}
\end{cases}
\end{gather}	
for all $\sigma\in [1,\infty)$.
\end{proposition}

\begin{proof}
Without loss of generality, we can assume that $f_1$ is positively homogeneous, $f_2$ is absolutely homogeneous, and $f_3,\ldots,f_n$ are 2-homogeneous.
Using a similar argument as in the proof of Corollary~\ref{C3.4}, we can find  nonzero points
$u_1\in X_1$ and $u_2\in X_2$ such that $f_1(u_1)=f_2(u_2)=f_2(-u_2)=\xi$  $(i=1,\ldots,n)$ for some $\xi>0$.
Let $x:=(u_1,0_{X_2},\ldots,0_{X_n})$ and $y :=(0_{X_1},u_2,0_{X_3},\ldots,0_{X_{n}})$.
By \eqref{p}, 
\begin{gather}\label{T3.11-2}
f_{\psi_p}(x)=f_{\psi_p}(y)=\xi,\;\;	
f_{\psi_p}(x+y)=f_{\psi_p}(x-y)=
\begin{cases}	
2^{\frac{1}{p}}\xi& \text{\rm if } p\in[1,\infty),\\
\xi & \text{\rm if } p=\infty.
\end{cases}
\end{gather}
Then
\begin{gather}
\mathcal{G}^{(\sigma)}_{d_{\psi_p}}(x,y)\overset{\eqref{G}}{=}\dfrac{f_{\psi_p}^\sigma(x+y)+f_{\psi_p}^\sigma(x-y)}{2^{\sigma-1}\left(f_{\psi_p}^\sigma(x)+f_{\psi_p}^\sigma(y)\right)}
\overset{\eqref{T3.11-2}}{=}
\begin{cases}	
2^{\frac{\sigma}{p}-\sigma+1}& \text{\rm if } p\in[1,\infty),\\
2& \text{\rm if } p=\infty.
\end{cases}\label{T3.5-1}
\end{gather}
Note that $f_{\psi_p}$ is  2-homogeneous for all $p\in[1,\infty]$.
Then
\begin{eqnarray*}
\mathcal{G}^{(\sigma)}_{d_{\psi_p}}(x+y,x-y)
&\overset{\eqref{G}}{=}& \dfrac{f_{\psi_p}^\sigma(2x)+f_{\psi_p}^\sigma(2y)}{2^{\sigma-1}(f_{\psi_p}^\sigma(x+y)+f_{\psi_p}^\sigma(x-y))}\\	&=&\dfrac{2(f_{\psi_p}^\sigma(x)+f_{\psi_p}^\sigma(y))}{f_{\psi_p}^\sigma(x+y)+f_{\psi_p}^\sigma(x-y)}\\
&\overset{\eqref{G}}{=}&\dfrac{2^{2-\sigma}}{\mathcal{G}^{(\sigma)}_{d_{\psi_p}}(x,y)},
\end{eqnarray*}
and consequently,
\begin{gather}
\label{T3.5-2}
\mathcal{G}_{d_{\psi_p}}^{(\sigma)}(x+y,x- y)=\dfrac{2^{2-\sigma}}{\mathcal{G}^{(\sigma)}_{d_{\psi_p}}(x,y)}=\begin{cases}	
2^{1-\frac{\sigma}{p}}& \text{\rm if } p\in[1,\infty),\\
2& \text{\rm if } p=\infty.
\end{cases}
\end{gather}	
By \eqref{T3.5-1} and \eqref{T3.5-2}, $C_{\text{\rm NJ}}^{(\sigma)}(d_{\psi_p})\ge 2$ if $p\in\{1,\infty\}$,
and
\begin{gather*}
C_{\text{\rm NJ}}^{(\sigma)}(d_{\psi_p})\ge
\max\{2^{\frac{\sigma}{p}-\sigma+1},2^{1-\frac{\sigma}{p}}\}=
\begin{cases}	
2^{\frac{\sigma}{p}-\sigma+1}& \text{\rm if } p\in(1,2],\\
2^{1-\frac{\sigma}{p}}& \text{\rm if } p\in(2,\infty).
\end{cases}
\end{gather*}
This completes the proof.
\end{proof}	

To derive  exact estimates for the constant in Proposition~\ref{T7.2}, we need the following property.
\begin{definition}\label{D3.4}
Let $d$ be a metric on a vector space $X$, and $f$ be given by \eqref{f}.
\begin{enumerate}[\rm (i)]
\item\label{D3.4-1}
Let $\al\in(1,2]$ and $\be\in\R$ with $\frac{1}{\al}+\frac{1}{\be}=1$.
The metric  $d$ is said to satisfy the $(\al,\be)$-\textit{Clarkson inequality} if 
\begin{gather}\label{D4.6-1}
f^\be(x+y)+f^\be(x-y)\le 2\left(f^\al(x)+f^\al(y)\right)^{\frac{\be}{\al}}
\end{gather}
for all $x,y\in X$. 
\item\label{D3.4-2}
The metric  $d$ is said to satisfy the Clarkson inequality if it satisfies the 
 $(\al,\be)$-Clarkson inequality for any $\al\in(1,2]$ and $\be\in\R$ with $\frac{1}{\al}+\frac{1}{\be}=1$.
\end{enumerate}
\end{definition}	

\begin{remark}\label{R4.14}
\begin{enumerate}[\rm (i)]
\item\label{R4.14-1}
Let $p \in(1,\infty)$.
Consider the vector spaces $L_p([0,1])$ of all complex-valued measurable functions on $[0,1]$ and $\ell_p$ of all complex sequences equipped with the norms
\begin{gather}\label{n}
\|x\|_p := \left(\int_0^1 |x(t)|^p \, dt \right)^{1/p}
\;\;\text{and}\;\;
\|y\|_p := \left( \sum_{k=1}^\infty |y_k|^p \right)^{1/p}
\end{gather}	
for all $x\in L_p([0,1])$ and $y:=(y_k)_{k \in \mathbb{N}}\in \ell_p$, respectively.
In~\cite[Theorem~2]{Cla36}, Clarkson proved that these norms satisfy the $(p,q)$-Clarkson inequality for every $p \in (1,2]$ and $q \in \mathbb{R}$ such that $\frac{1}{p} + \frac{1}{q} = 1$.
This result plays a crucial role in proving the uniform convexity of these spaces \cite[Corollary, p. 403]{Cla36} and in computing the conventional von Neumann-Jordan constant of these norms~\cite[Theorem, p. 114]{Cla37}.
\item\label{R4.14-2}
For a general Banach space, it was shown by Kato and Takahashi~\cite[Theorem~2.9]{KatTak97} that a norm satisfies the $(\alpha,\beta)$-Clarkson inequality if and only if its dual norm satisfies the  inequality with the same $(\al,\be)$. 
Another proof of this result can be found in~\cite[Theorem~2]{Cho01}. 
A recent study~\cite{Cuo25} uses the norm version of the property in Definition~\ref{D3.4} to compute the conventional von Neumann-Jordan constant of $p$-norms on a product of vector spaces.
\end{enumerate}
\end{remark}

The following result establishes a characterization of the Clarkson inequality property.
\begin{proposition}\label{P4.14}
Let $d$ be a metric on a vector space $X$, $f$ be given by \eqref{f}, $\al\in(1,2]$ and $\be\in\R$ with $\frac{1}{\al}+\frac{1}{\be}=1$.
Suppose that $f$ is 2-homogeneous.
Then $f$ satisfies the $(\al,\be)$-Clarkson inequality if and only if
\begin{gather}\label{P4.14-1}
 2\left(f^{\be}(x)+f^{\be}(y)\right)^{\frac{\al}{\be}}\le f^{\al}(x+y)+f^{\al}(x-y)
\end{gather}
for all $x,y\in X$.
\end{proposition}	

\begin{proof}
Let $f$ be 2-homogeneous. 
Suppose that $f$ satisfies the $(\al,\be)$-Clarkson inequality.
Let $x,y\in X$.
By Definition~\ref{D3.4}\eqref{D3.4-1} and Remark~\ref{R1.2}\eqref{R1.2-2},
\begin{align*}
f^{\be}(x)+f^{\be}(y)
\le 2\left(f^\al\left(\frac{x+y}{2}\right)+f^\al\left(\frac{x-y}{2}\right)\right)^{\frac{\be}{\al}}
=2^{1-\be}(f^{\al}(x+y)+f^{\al}(x-y))^{\frac{\be}{\al}},
\end{align*}
and consequently,
\begin{gather*}
2^{\frac{\al(\be-1)}{\be}}(f^{\be}(x)+f^{\be}(y))^{\frac{\al}{\be}} \le f^{\al}(x+y)+f^{\al}(x-y).
\end{gather*}
Note that $\frac{\al(\be-1)}{\be}=1$.
Thus,
condition \eqref{P4.14-1} is satisfied.
Conversely, suppose that condition \eqref{P4.14-1} holds true.
Let $x,y\in X$.
By  \eqref{P4.14-1} and Remark~\ref{R1.2}\eqref{R1.2-2},
\begin{align*}
f^{\al}(x)+f^{\al}(y)
\ge 	
2\left(f^\be\left(\frac{x+y}{2}\right)+f^\be\left(\frac{x-y}{2}\right)\right)^{\frac{\al}{\be}}
=2^{1-\al}(f^\be(x+y)+f^\be(x-y))^{\frac{\al}{\be}},
\end{align*}
and consequently,
\begin{gather*}
f^\be(x+y)+f^\be(x-y)\le 2^{\frac{\be(\al-1)}{\al}} \left(f^\al(x)+f^\al(y)\right)^{\frac{\be}{\al}}.
\end{gather*}
Note that $\frac{\be(\al-1)}{\al}=1$.
Thus, condition \eqref{D4.6-1} is satisfied.
By Definition~\ref{D3.4}\eqref{D3.4-1}, $f$ satisfies the $(\al,\be)$-Clarkson inequality.
\end{proof}	

\begin{remark}
A normed version of Proposition~\ref{P4.14} can be found in \cite[Theorem~2]{Cla36}. 
\end{remark}

The following lemma establishes a relation between the Clarkson inequality property of the $p$-metric and that of its components.
\begin{lemma}\label{L4.14}
Let $p\in(1,2]$, $q\in\R$ with $\frac{1}{p}+\frac{1}{q}=1$, and $f_{\psi_p}$ be given by \eqref{p}.
Then $f_1,\ldots,f_n$ satisfy the $(p,q)$-Clarkson inequality if and only if $f_{\psi_{p}}$ the $(p,q)$-Clarkson inequality.
\end{lemma}	

\begin{proof}
Suppose that $f_1,\ldots,f_n$ satisfy the $(p,q)$-Clarkson inequality. 
Let $x:=(x_1,\ldots,x_n), y:=(y_1,\ldots,y_n)\in \mathrm{P}$. 
Let $a_i:=f^q_i(x_i+y_i)$ and $b^q_i:=f_i(x_i-y_i)$ $(i=1,\ldots,n)$.
By %\todo{Definition~\ref{D3.4}}
Definition \ref{D3.4}\eqref{D3.4-1},
\begin{equation}\label{L4.14-1}
a_i+b_i\le 2(f_i^p(x_i)+f_i^p(y_i))^{\frac{1}{p-1}}\;\; (i=1,\ldots,n).
\end{equation}
By \eqref{p},
\begin{gather}\label{L4.14-2}
f_{\psi_{p}}^q(x+y)=\left(\sum_{i=1}^{n}f_i^p(x_i+y_i)\right)^\frac{q}{p}
= \left(\sum_{i=1}^{n}(f_i^q(x_i+ y_i))^{p-1}\right)^\frac{q}{p}=\left(\sum_{i=1}^{n}a_i^{p-1}\right)^\frac{1}{p-1},\\ \label{L4.14-3}
f_{\psi_{p}}^q(x-y)=\left(\sum_{i=1}^{n}f_i^p(x_i-y_i)\right)^\frac{q}{p}=\left(\sum_{i=1}^{n}(f_i^q(x_i- y_i))^{p-1}\right)^\frac{q}{p}=\left(\sum_{i=1}^{n}b_i^{p-1}\right)^\frac{1}{p-1}.
\end{gather}
Then
\begin{eqnarray*}
f_{\psi_{p}}^q(x+y)+f_{\psi_{p}}^q(x-y)
&\overset{\eqref{L4.14-2},\eqref{L4.14-3}}{=}&\left(\sum_{i=1}^{n}a_i^{p-1}\right)^\frac{1}{p-1}+\left(\sum_{i=1}^{n}b_i^{p-1}\right)^\frac{1}{p-1}\\
&\le&\left(\sum_{i=1}^{n}\left(a_i+b_i\right)^{p-1}\right)^\frac{1}{p-1}\;\;
(\text{Lemma}~\ref{L1.5}\eqref{L1.5-3})\\
&\overset{\eqref{L4.14-1}}{\le}&\left(\sum_{i=1}^{n}2^{p-1}\left(f_i^p(x_i)+f_i^p(y_i)\right)\right)^\frac{1}{p-1}\\
&=&2\left(\sum_{i=1}^{n}\left(f_i^p(x_i)+f_i^p(y_i)\right)\right)^\frac{q}{p}\\
&\overset{\eqref{p}}{=}&2\left(f_{\psi_p}^p(x)+f_{\psi_p}^p(y)\right)^\frac{q}{p}.
\end{eqnarray*}
Thus, $f_{\psi_{p}}$ satisfies the $(p,q)$-Clarkson inequality.
For the converse implication, it suffices to observe that for each $i \in \{1,\ldots,n\}$ and any $x_i, y_i \in X_i$, we have
$f_{\psi_p}(x) = f_i(x_i)$, $f_{\psi_p}(y) = f_i(y_i)$ and $f_{\psi_p}(x+y)= f_i(x_i + y_i)$ thanks to \eqref{p}, where $x:=(0_{X_1},\ldots, 0_{X_{i-1}}, x_i, 0_{X_{i+1}},\ldots,0_{X_n})$ and $y:=(0_{X_1},\ldots, 0_{X_{i-1}}, y_i, 0_{X_{i+1}},\ldots,0_{X_n})$.
This completes the proof.
\end{proof}	

\begin{remark}
\begin{enumerate}[\rm (i)]
\item		
The proof of the direct implication in Lemma~\ref{L4.14} employs the technique introduced by Clarkson~\cite[Theorem~2]{Cla36}.
In that work, the author first proved that the standard norm on $\mathbb{C}$ satisfies the Clarkson inequality and then used this fact to show that the Clarkson inequality property holds for the norms given by~\eqref{n}.
\item
Observe that the norm on $\ell_p$ in \eqref{n} is defined on a countable product of copies of the vector space $\mathbb{C}$, while Lemma~\ref{L4.14} is formulated in terms of a finite product space. 
Nevertheless, the lemma follows a somewhat different approach since it allows $f_1,\ldots,f_n$ to be general functions defined on possibly different vector spaces.
\end{enumerate}
\end{remark}

\begin{theorem}\label{T2.11}
Let the assumptions in Proposition~\ref{T7.2} be satisfied.
The following assertions hold.
\begin{enumerate}[\rm (i)]
\item\label{T2.11-1}
Let $p\in (1,2]$ and $q\in\R$ with $\frac{1}{p}+\frac{1}{q}=1$. 
If $f_1,\ldots,f_n$ satisfy the $(p,q)$-Clarkson inequality, then
$C_{\text{\rm NJ}}^{(\sigma)}(d_{\psi_p})=2^{\frac{\sigma}{p}-\sigma+1}$ for all $\sigma\in [p,q]$.
\item\label{T2.11-2}
Let $p\in(2,\infty)$ and $q\in\R$ with $\frac{1}{p}+\frac{1}{q}=1$. If $f_1,\ldots,f_n$ satisfy the $(q,p)$-Clarkson inequality, then
%\todo{$C_{\text{\rm NJ}}^{(\sigma)}(d_{\psi_p})$}
$C_{\text{\rm NJ}}^{(\sigma)}(d_{\psi_p})=2^{1-\frac{\sigma}{p}}$ for all
$\sigma\in [q,p].$
\item\label{T2.11-3}
If $f_1,\ldots,f_n$ are subadditive and even, then $C_{\text{\rm NJ}}^{(\sigma)}(d_{\psi_1})=C_{\text{\rm NJ}}^{(\sigma)}(d_{\psi_\infty})=2$
for all $\sigma\in[1,\infty)$.
\end{enumerate}	
\end{theorem}

\begin{proof}	
\begin{enumerate}[\rm (i)]
\item	
Let $p\in (1,2]$, $q\in\R$ with $\frac{1}{p}+\frac{1}{q}=1$, and $\sigma\in [p,q]$.
By \eqref{T3.10-1}, it suffices to prove that $C_{\text{NJ}}^{(\sigma)}(d_{\psi_p})\le 2^{\frac{\sigma}{p}-\sigma+1}$.
%Note that $p\le 2\le q$.
In view of Lemma~\ref{L4.14}, the function $f_{\psi_p}$ satisfies the $(p,q)$-Clarkson inequality.
Let $x,y\in\mathrm{P}$.
By Definition~\ref{D3.4}\eqref{D3.4-1},
\begin{gather}\label{T3.5-3}
f^q_{\psi_p}(x+y)+f^q_{\psi_p}(x-y)\le 2\left(f^p_{\psi_p}(x)+f^p_{\psi_p}(y)\right)^{\frac{q}{p}}.
\end{gather}	
Note that $\frac{\sigma}{q},\frac{p}{\sigma}\in(0,1]$.
Then
\begin{eqnarray*}
f^\sigma_{\psi_p}(x+y)+f^\sigma_{\psi_p}(x-y)
&=&\left(f^q_{\psi_p}(x+y)\right)^{\frac{\sigma}{q}}+\left(f^q_{\psi_p}(x-y)\right)^{\frac{\sigma}{q}}\\
&\le&2^{1-\frac{\sigma}{q}}\cdot(f^q_{\psi_p}(x+y)+f^q_{\psi_p}(x-y))^{\frac{\sigma}{q}}\;\;(\text{Lemma~\ref{L1.5}\eqref{L1.5-1}})\\
&\overset{\eqref{T3.5-3}}{\le}&2\left(f^p_{\psi_p}(x)+f^p_{\psi_p}(y)\right)^{\frac{\sigma}{p}}\\
&=&2\left[\left(f^\sigma_{\psi_p}(x)\right)^{\frac{p}{\sigma}}+\left(f^\sigma_{\psi_p}(y)\right)^{\frac{p}{\sigma}}\right]^{\frac{\sigma}{p}}\\
&\le& 2\left[2^{1-\frac{p}{\sigma}}\cdot\left(f^\sigma_{\psi_p}(x)+f^\sigma_{\psi_p}(y)\right)^{\frac{p}{\sigma}}\right]^{\frac{\sigma}{p}}\;\;(\text{Lemma~\ref{L1.5}\eqref{L1.5-1}})\\
&=&2^{\frac{\sigma}{p}}(f^\sigma_{\psi_p}(x)+f^\sigma_{\psi_p}(y))\\
&=&2^{\sigma-1}(f^\sigma_{\psi_p}(x)+f^\sigma_{\psi_p}(y))\cdot2^{\frac{\sigma}{p}-\sigma+1}.
\end{eqnarray*}	
By \eqref{G} and \eqref{C1}, $C_{\text{NJ}}^{(\sigma)}(d_{\psi_p})\le 2^{\frac{\sigma}{p}-\sigma+1}$.
\item
Let $p\in(2,\infty)$, $q\in\R$ with $\frac{1}{p}+\frac{1}{q}=1$, and $\sigma\in[q,p]$.
By \eqref{T3.10-1}, it suffices to prove that $C_{\text{NJ}}^{(\sigma)}(d_{\psi_p})\le 2^{1-\frac{\sigma}{p}}$.
Note that $q\in(1,2)$.
Using similar arguments as in the proof of \eqref{T2.11-1} with $p$ and $q$ interchanged, we obtain 
$C_{\text{NJ}}^{(\sigma)}(d_{\psi_p})\le 2^{\frac{\sigma}{q}-\sigma+1}= 2^{1-\frac{\sigma}{p}}$.
\item
Let $p\in\{1,\infty\}$, and $\sigma\in [1,\infty)$.
In view of \eqref{T3.10-1}, we have
$C_{\text{NJ}}^{(\sigma)}(d_{\psi_p})\ge 2$.
Under the assumption made, $f_{\psi_p}$ is subadditive and even.
By %\todo{Theorem~\ref{T2.5}\eqref{T2.5-3}}
Theorem~\ref{T2.5}\eqref{T2.5-3}, $C_{\text{NJ}}^{(\sigma)}(d_{\psi_p})\le 2$.
This proves the assertion.
\end{enumerate}
The proof is complete.
\end{proof}	

In view of Theorem~\ref{T2.11}\eqref{T2.11-3}, the generalized von Neumann-Jordan constants of $d_{\psi_1}$ and $d_{\psi_\infty}$ coincide.
The following statement shows that equalities also hold for the other pairs of metrics.
\begin{corollary}\label{C7.3}
Let $p\in(1,\infty)$, and $q\in\R$ with $\frac{1}{p}+\frac{1}{q}=1$.
Suppose that the assumptions in Proposition~\ref{T7.2} are satisfied.
If $f_1,\ldots,f_n$ satisfy the Clarkson inequality, then
\begin{gather*}
C_{\text{\rm NJ}}^{(\sigma)}(d_{\psi_p})=C_{\text{\rm NJ}}^{(\sigma)}(d_{\psi_q})=
\begin{cases}	
2^{\frac{\sigma}{p}-\sigma+1}& \text{\rm if } p\in(1,2]\;\; \text{\rm and }\sigma\in\left[p,q\right],\\
2^{1-\frac{\sigma}{p}}& \text{\rm if } p\in(2,\infty)\;\;\text{\rm and }  \sigma\in\left[q,p\right].
\end{cases}
\end{gather*}
\end{corollary}

\begin{proof}
We use Theorem~\ref{T2.11}\eqref{T2.11-1} and \eqref{T2.11-2} in the following arguments.	
If $p\in(1,2]$, then
\begin{gather*}
C_{\text{\rm NJ}}^{(\sigma)}(d_{\psi_p})	
\overset{\eqref{T2.11-1}}{=}
2^{\frac{\sigma}{p}-\sigma+1}=2^{1-\frac{\sigma}{q}}\overset{\eqref{T2.11-2}}{=}
C_{\text{\rm NJ}}^{(\sigma)}(d_{\psi_q})
\end{gather*}	
for all $\sigma\in[p,q]$.	
If $p\in(2,\infty)$, then
\begin{gather*}
C_{\text{\rm NJ}}^{(\sigma)}(d_{\psi_p})\overset{\eqref{T2.11-2}}{=}2^{1-\frac{\sigma}{p}}=
2^{\frac{\sigma}{q}-\sigma+1}\overset{\eqref{T2.11-1}}{=}C_{\text{\rm NJ}}^{(\sigma)}(d_{\psi_q})
\end{gather*}	
for all $\sigma\in[q,p]$.	
The proof is complete.
\end{proof}	

\begin{remark}
\begin{enumerate}[\rm (i)]
\item		
For every $p \in (1,\infty)$, the admissible interval for $\sigma$ in Corollary~\ref{C7.3} always contains $2$. 
When $X_1=\cdots=X_n$ is a normed space and $\sigma=2$,
this corollary recaptures the corresponding results in~\cite{Cuo25,SaiKatTak00,Cla37}.
\item
When $p=2$, the admissible interval is the singleton $\{2\}$, and consequently,
$C_{\text{\rm NJ}}^{(2)}(d_{\psi_2})=1$.
This extends the classical result of Clarkson and von Neumann from the normed to the metric setting \cite{JorNeu35,SaiKatTak00+}.
\end{enumerate}
\end{remark}	

\section*{Disclosure statement}
The authors report there are no competing interests to declare.

\section*{Data availability statement}
Data sharing is not applicable to this article as no new data were created or analyzed in this study.


\begin{thebibliography}{99}
	\bibitem{JorNeu35}
	Jordan~P, Neumann~JV. On inner products in linear, metric spaces. Annals of Mathematics. 1935;\hspace{0pt}36(3): 719-723.
	
	\bibitem{Cla36}
	Clarkson~JA. Uniformly convex spaces. Transactions of the American Mathematical
	Society. 1936;\hspace{0pt}40(3):396-414.
	
	\bibitem{Cla37}
	Clarkson~JA. The von {Neumann}--{Jordan} constant for the {Lebesgue} spaces.
	Annals of Mathematics. 1937;\hspace{0pt}38:114-115.
	
	 \bibitem{Bea82}
	Beauzamy B. Introduction to Banach spaces and their geometry. North-Holland Mathematics Studies, Vol. 68. North-Holland Publishing Company. 1982.
	
	\bibitem{JohLin01}
	Johnson WB, Lindenstrauss J. Handbook of the geometry of Banach spaces. Vol. 1. Elsevier. 2001.
	
	  \bibitem{KatTak98}
	Kato M, Takahashi Y. Von Neumann-Jordan constant for Lebesgue--Bochner spaces. Journal of Inequalities and Applications. 1998;2:89--97.
	
	\bibitem{KatMal01}
	Kato M, Maligranda L. On James and Jordan--von Neumann constants of Lorentz sequence spaces. Journal of Mathematical Analysis and Applications. 2001;258:457-465.
	
	\bibitem{Sae10}
	Saejung S. Another look at Cesàro sequence spaces. Journal of Mathematical Analysis and Applications. 2010;366:530--537.
	
	%\bibitem{Zho25}	Zhou H, Liu Q, Wang Y, Liang M, Fu L. Skew von Neumann constant in weak Orlicz spaces and weak Lebesgue spaces. 2025; arXiv:2508.06999.
	
  \bibitem{CuiHuaHudKac15}
	Cui~Y, Huang~W, Hudzik~H, Kaczmarek~R.
	Generalized von Neumann-Jordan constant and its relationship to the fixed point property. Fixed Point Theory and Applications.
	2015;\hspace{0pt}2015(1),40.
	
	\bibitem{WanCuiZha15}
	Wang~X, Cui~Y, Zhang~C.
	The generalized von Neumann-Jordan constant and normal structure in Banach spaces. Annals of Functional Analysis.
	2015;\hspace{0pt}6(4):206-214.
	
	 \bibitem{YanLi10}
	Yang C, Li H. An inequality between Jordan-von Neumann constant and James constant. Applied Mathematics Letters. 2010;23:277-281.
	
%	\bibitem{Zuo21}	Zuo Z.-F., Wang L.-W., Zhao Y.-X., Wu Y.-Q. On the generalized von Neumann–Jordan type constant for some concrete Banach spaces. Mathematical Inequalities and Applications. 2021;24:597–615.
	
%	\bibitem{LiYanYan24}	Li H., Yang X., Yang C. On $(n,p)$–th von Neumann–Jordan constants for Banach spaces. Mathematical Inequalities and Applications. 2024;27:583–600.
	
%	\bibitem{TakKat14}
%	Takahashi Y, Kato M. On a new geometric constant related to the modulus of smoothness of a Banach space. Acta Mathematica Sinica, English Series. 2014;30:1526--1538.
	
	\bibitem{DhoPirSae03}
	Dhompongsa S, Piraisangjun P, Saejung S. Generalised Jordan-von Neumann constants and uniform normal structure. Bulletin of the Australian Mathematical Society. 2003;67:225--240.
	
	\bibitem{Wan10}
	Wang F. On the James and von Neumann-Jordan constants in Banach spaces. Proceedings of the American Mathematical Society. 2010;138:695-701.
	
	\bibitem{WanPan09}
	Wang F, Pang B. Some inequalities concerning the James constant in Banach spaces. Journal of Mathematical Analysis and Applications. 2009;353:305-310.
	
%	\bibitem{DinSan25}
%	Dinarvand M, Sanatpour AH. Normal structure and generalized von Neumann–Jordan type constant. Filomat. 2025;39:11701--11716.
	
	\bibitem{KatTak97+}
	Kato~M, Takahashi~Y. On the von {Neumann}--{Jordan} constant for {Banach}
	spaces. Proceedings of the American Mathematical Society.
	1997;\hspace{0pt}125(4):1055-1062.
	
	  \bibitem{TakKat98}
	Takahashi Y, Kato M. Von Neumann-Jordan constant and uniformly non-square Banach spaces. Nihonkai Mathematical Journal. 1998;9:155-169.
	
	\bibitem{JimLloSae06}
	Jiménez-Melado A, Llorens-Fuster E, Saejung S. The von Neumann–Jordan constant, weak orthogonality and normal structure in Banach spaces. Proceedings of the American Mathematical Society. 2006;134:355-364.
	
	 \bibitem{DhoDomKaeKaePan06}
	Dhompongsa~S, Domínguez~Benavides~T, Kaewcharoen~A, Kaewkhao~A, Panyanak~B.
	The Jordan--von Neumann constants and fixed points for multivalued nonexpansive mappings.
	Journal of Mathematical Analysis and Applications. 2006;\hspace{0pt}320(2):916-927.
	
	\bibitem{Sae06}
	Saejung S. On James and von Neumann-Jordan constants and sufficient conditions for the fixed point property. Journal of Mathematical Analysis and Applications. 2006;323:1018-1024.
	
	\bibitem{Cuo25}
	Cuong ND. Primal and dual characterizations of sign--symmetric norms.
	2025; arXiv:2504.19642.
	To appear in \emph{Positivity}.
	
	\bibitem{SaiKatTak00}
	Saito KS, Kato M, Takahasi Y. Von Neumann-Jordan constant of absolute normalized norms on {$\mathbb{C}^2$}. Journal of Mathematical Analysis and Applications. 2000;\hspace{0pt}244(2):515--532.
	
	\bibitem{SaiKatTak00+}
	Saito~KS, Kato~M, Takahashi~Y. Absolute norms on $\mathbb{C}^n$.
	Journal of Mathematical Analysis and Applications. 2000;\hspace{0pt}252(2):879-905.
	
	\bibitem{IkeKat14}
	Ikeda T, Kato M. Notes on von Neumann-Jordan and James constants for absolute norms on $\mathbb{R}^2$. Mediterranean Journal of Mathematics. 2014;11:633-642.
	
	\bibitem{Rol84}
	Rolewicz S. Metric linear spaces. PWN-Polish Scientific Publishers. Warsaw. 1985.
	
	 \bibitem{Don21}
	Dontchev AL. Lectures on variational analysis. Applied Mathematical Sciences. Springer Cham. 2021.
	
	\bibitem{Kak36}
	Kakutani S. Über die metrisation der topologischen gruppen. Proceedings of the Imperial Academy. 1936;12(4):82-84.
	
	\bibitem{Mab03}
	Mabizela S. Characterization of best approximations in metric linear spaces. Analysis in Theory and Applications. 2003;19(2): 121-129.
	
	\bibitem{Alb79}
	Albinus G. Approximation in metric linear spaces. Banach Center Publ. 1979;4:7-18.
	
	\bibitem{Yos95}
	Yosida K. Functional analysis. Classics in Mathematics. Springer Berlin Heidelberg. 1995.
	
	\bibitem{Rol72}
	Rolewicz S. Open problems in theory of metric linear spaces. Mémoires de la Société Mathématique de France. 1972;31-32:327-334.
	
	 \bibitem{SinNar20}
	Singh J, Narang TD. Convex linear metric spaces are normable. Journal of Analysis. 2020;28:705-709.
	
	\bibitem{AmiKho24}
	Amini-Harandi A, Khosravani S. On the calculations of the generalized von Neumann–Jordan constants $C^{(p)}_{NJ}(L^r(\Omega,\Sigma,\mu))$ and $\tilde{C}^{(p)}_{NJ}(L^r(\Omega,\Sigma,\mu))$. Mathematical Inequalities and Applications. 2024;27:659-667.
	
	\bibitem{RahGun21}
	Rahman H, Gunawan H. Generalized von Neumann-Jordan constant for Morrey spaces and small Morrey spaces. Australasian Journal of Mathematical Analysis and Applications. 2021;18(1):Article 17, 1-7.
    
    \bibitem{Mar09}
    Kuczma~M. An introduction to the theory of functional equations and
    inequalities: Cauchy’s equation and Jensen’s inequality. Basel:
    Birkhäuser Basel; 2009. Texts and Readings in Mathematics.
	
	 \bibitem{Pen13}
	Penot JP. Calculus without derivatives. Graduate Texts in Mathematics, Vol. 266. Springer New York. 2013.
	
	\bibitem{Jen06}
	Jensen~JLWV. Sur les fonctions convexes et les inégalités entre les valeurs
	moyennes. Acta Mathematica. 1906;\hspace{0pt}30(1):175-193.
	
	\bibitem{TamTsu21}
	Tamura~A, Tsurumi~K.  Directed discrete midpoint convexity. 
	Japan Journal of Industrial and Applied Mathematics. 2021;\hspace{0pt}38(1):1-37.

	\bibitem{YanWan06}
    Yang~C, Wang~F. On a new geometric constant related to the von Neumann-Jordan constant. Journal of mathematical analysis and applications. 2006;\hspace{0pt}324(1):555--565.
	
	 \bibitem{BonDun73}
	Bonsall FF, Duncan J. Numerical ranges II. London Mathematical Society Lecture Note Series. Cambridge University Press. 1973.
	
    \bibitem{HieTamCuo26}
	Hieu DH, Tam VM, Cuong ND. Metric constructions and fixed point theorems in product spaces. 2026; arXiv:2601.15907.
	
	\bibitem{Cuo26}
	Cuong ND. Dual characterizations of norm minimization problems. 
	 2026; arXiv:2601.08153.
	To appear in \textit{Optimization}. 
	
	\bibitem{TanOhwSai14}
	Tanaka R, Ohwada T, Saito KS. Geometric constants and characterizations of inner product spaces. Mathematical Inequalities and Applications. 2014;17:513-520.
	
	\bibitem{Cao03}
	Cao~HX. Von {Neumann}--{Jordan} constants of absolute normalized norms on
	{$\mathbb{C}^n$}. Acta Mathematica Sinica. 2003;\hspace{0pt}19(3):507-512.
	
	\bibitem{KatTak97}
	Kato~M, Takahashi~Y. Type, cotype constants and Clarkson's inequalities for Banach Spaces.
	Mathematische Nachrichten. 1997;\hspace{0pt}186:187-196.
	
	\bibitem{Cho01}
	Cho~CM. A note on Clarkson's inequalities.
	Bulletin of the Korean Mathematical Society. 2001;\hspace{0pt}38(4):657-662.

\end{thebibliography}
\end{document}